%% file: main.tex
\documentclass[12pt]{amsart}
\usepackage[top=1in, bottom=1.25in, left=1in, right=1in]{geometry}
\usepackage{amssymb}
\usepackage{amsmath}
\usepackage{amscd}
\usepackage[pdftex]{graphicx}
\usepackage[all]{xy}

\usepackage{xcolor}

\newcommand{\map}[3]{#1\colon #2\to #3}
\newtheorem{theorem}{Theorem}[section]
\newtheorem{lemma}[theorem]{Lemma}
\newtheorem{prop}[theorem]{Proposition}

\newtheorem{cor}[theorem]{Corollary}

\theoremstyle{definition}
\newtheorem{defi}[theorem]{Definition}
\theoremstyle{remark}
\newtheorem{rem}[theorem]{Remark}
\newtheorem{example}[theorem]{Example}
\newcommand{\mcg}{\mathrm{Mod}}

\newcommand{\pants}{\mathcal{P}}
\newcommand{\Circ}{\mathcal{C}}
\newcommand{\Tri}{\mathcal{T}}
\newcommand{\Aut}{\mathrm{Aut}}
\newcommand{\curv}{\mathrm{Curv}}
\newcommand{\farey}{\mathcal{F}}
\newcommand{\St}{\mathrm{St}}
\newcommand{\Lk}{\mathrm{Lk}}

\newcommand{\Z}{\mathbb{Z}}

\newlength{\customwidth}
\usepackage{mathtools}
\newcommand{\mathdash}{\relbar\mkern-4mu\relbar}
\newcommand{\mathdashR}{\relbar\mkern-7mu\rightarrow}
\newcommand{\mathdashL}{\leftarrow\mkern-7mu\relbar}

\newcommand{\plainMove}[2]{\,\ooalign{$#2$\cr
    \hidewidth\raisebox{1.15ex}{\text{\tiny {\ #1}}} 
    \hidewidth\cr}\,}

\newcommand{\mvi}{\plainMove{\,$i$}{\mathdash}} 
\newcommand{\mvI}{\plainMove{\;1}{\mathdash}}  
\newcommand{\mvII}{\plainMove{\,2}{\mathdash}}  
\newcommand{\mvIII}{\plainMove{\,3}{\mathdash}}  
\newcommand{\mvIV}{\plainMove{4}{\mathdashR}}  
\newcommand{\mvIVL}{\plainMove{4}{\mathdashL}}  
\newcommand{\mv}{\plainMove{\ }{\mathdash}} 

\newcommand{\fullMv}[4]{#1\!\colon #2 #3 #4}

\newcommand{\multi}[1]{\underline{#1}}

\DeclareMathOperator{\im}{im}

\numberwithin{equation}{section}
\title[Automorphisms of the pants graph]{Automorphisms of the pants graph of a nonorientable surface}
\author{Micha{\l} Stukow \and B{\l}a\.zej Szepietowski}
\address{Institute of Mathematics, Faculty of Mathematics, Physics and Informatics, University of Gda\'nsk, 80-308 Gda\'nsk, Poland} 
\email{michal.stukow@ug.edu.pl}
\email{blazej.szepietowski@ug.edu.pl}
\begin{document}
\begin{abstract}
We prove that, except in certain low-complexity cases, the automorphism group of the graph of pants decompositions of a nonorientable surface is isomorphic to the mapping class group of that surface.
\end{abstract}

\maketitle
\input{Introduction}
\input{Preliminaries}
\input{Pants}
\input{F_graphs}
\input{Circuits_new}

\input{Edge_in_circuit_new}
\input{aut_pants_new}
\input{Iso_pants}
\input{aut_curve}
\input{N30}

\medskip
\noindent{\bf Acknowledgments.} We thank our colleagues Marta Le\'sniak and Jakub Szmelter-Tomczuk for fruitful discussions. Special thanks go to  Jakub for finding Example \ref{ex:5gon}.


\input{bibliografia}
\end{document}

%% file: Introduction.tex
\section{Introduction}
Let $S$ be a compact surface,  possibly nonorientable and with
boundary. We define the mapping class group $\mcg(S)$ of $S$ to be the group of isotopy classes
of all homeomorphisms of $S$. If S is orientable then $\mcg(S)$ is usually called the extended
mapping class group and we denote by $\mcg^+(S)$ the subgroup of $\mcg(S)$ consisting of orientation preserving mapping classes. 

A connected orientable (resp. nonorientable) surface of genus $g$ with $b$ boundary components will be denoted by $S_{g,b}$ (resp. $N_{g,b}$). If $b=0$ we drop it from the notation. Note that
$N_{g,b}$ is obtained from $S_{0,g+b}$ by gluing $g$ M\"obius bands (also called crosscaps) along $g$ distinct boundary components of $S_{0,g+b}$.

One of the most important breakthroughs in the study of mapping class groups $\mcg^+(S_{g,b})$ of oriented surfaces was the introduction of the so-called \emph{cut-system complex} by Hatcher and Thurston in \cite{HatcherThurston}. They proved that this 2-dimensional complex is connected and simply-connected. Moreover, the action of $\mcg^+(S_{g,b})$ on cut-system complex is cocompact, so the full finite presentation of $\mcg^+(S_g)$ can be derived from the presentations of the stabilizers of vertices and edges. This program was initiated by Hatcher and Thurston in \cite{HatcherThurston}, completed by Harer \cite{Harer}, and later greatly simplified by Wajnryb \cite{WajnrybPres,WajnrybSimp} -- see \cite{WajnrybSimp} for more details. 
\subsection{The pants complex}
In Appendix of \cite{HatcherThurston} Hatcher and Thurston introduced another important concept -- the graph $\pants^1(S_{g,b})$ of maximal cut systems of an oriented surface.  The vertices of this graph are pants decompositions of $S_{g,b}$ and edges correspond to two types of elementary moves (see Section \ref{sec:pants:graph}). Later, Hatcher, Lochak and Schneps \cite{HatcherPants,HatcherLochak} proved that $\pants^1(S_{g,b})$ is connected and added five types of 2-dimensional cells to turn the graph $\pants^1(S_{g,b})$ into 2-dimensional complex $\pants(S_{g,b})$ and they proved that $\pants(S_{g,b})$ is simply-connected. This complex is now called the \emph{pants complex} of $S_{g,b}$.   

There are a number of interesting results concerning $\pants(S_{g,b})$. Brock \cite{Brock} proved that $\pants^1(S_{g,b})$ is quasi-isometric to Teichm\"uller space with the Weil-Petersson metric. Margalit \cite{Mar} proved the following characterization of $\Aut(\pants(S_{g,b}))$.
\begin{theorem}\label{Th:Margalit}
The natural map $\map{\theta}{\mcg(S_{g,b})}{\Aut(\pants(S_{g,b}))}$
is surjective. Moreover, $\ker\theta\cong\Z_2$ for $S\in \{S_{1,1}, S_{1,2}, S_{2}\}$, $\ker\theta\cong\Z_2\oplus \Z_2$ for $S = S_{0,4}$, $\ker\theta = \mcg(S)$ for $S=S_{0,3}$, and $\ker\theta$ is trivial otherwise.
\end{theorem}
Later, this result was extended by Aramayona \cite{Ara} to injective simplicial and locally-injective simplicial maps of $\pants(S)$.

The proof of Theorem \ref{Th:Margalit} is based on a similar result for another simplicial complex associated to a surface, namely the \emph{curve complex} $\curv(S)$. The vertices of $\curv(S)$ correspond to homotopy classes of nontrivial simple closed curves in $S$, and $k$-simplexes correspond to sets of $k+1$ vertices that may be represented as a multicurve in $S$ -- see Section \ref{sec:preli} for a proper definition. One of the fundamental properties of $\curv(S)$ is the following
\begin{theorem}[Ivanov \cite{IvanovCs}, Korkmaz \cite{KorkmazCs}, Luo \cite{Luo}] \label{AutCs}
If $S\neq S_{0,3}$ is an orientable surface with $\chi(N)<0$ and 
$\map{\eta}{\mcg(S)}{\Aut(\curv(S))}$ is the natural map, then
\begin{enumerate}
    \item $\eta$ is surjective when $S\neq S_{1,2}$ and $\im \theta$ is a proper subgroup of $\Aut(\curv(S))$ that preserves the set of nonseparating curves for $S=S_{1,2}$.
    \item $\ker\eta\cong\Z_2\oplus\Z_2$ for $S=S_{0,4}$, $\ker\eta\cong\Z_2$ for $S\in\{S_{1,1}, S_{1,2},S_{2}$\}, and $\ker\eta$ is trivial otherwise.
\end{enumerate}
\end{theorem}
In contrast to the orientable case, much less is known about the pants complex of a nonorientable surface. 
The most significant result to date is due to Papadopoulos and Penner \cite{PP}, who introduced two new elementary moves between pants decompositions of a nonorientable surface (see Section \ref{sec:pants:graph}) and proved that the resulting pants graph $\pants(N_{g,b})$ is connected. 
\subsection{The main theorems}
The main goal of this paper is to extend Theorem \ref{Th:Margalit} to the case of a nonorientable surface (Theorem \ref{main:aut_pants}). However, as a first step we need the following  nonorientable version of Theorem \ref{AutCs}. 
\begin{theorem}\label{main:aut_curv}
If $N$ is a nonorientable surface with $\chi(N)<0$ and 
$\eta\colon\mcg(N)\to\Aut(\curv(N))$ is the natural map, then
\begin{enumerate}
    \item $\eta$ is surjective;
    \item $\ker\eta\cong\Z_2\oplus\Z_2$ for $N=N_{1,2}$, $\ker\eta\cong\Z_2$ for $N\in\{N_{2,1}, N_3$\}, and $\ker\eta$ is trivial otherwise.
\end{enumerate}
\end{theorem}
The above theorem was proved by Atalan and Korkmaz \cite{AK} for nonorientable surfaces $N_{g,b}$ with $g+b\ne 4$, the case $N=N_{1,3}$ was settled in \cite{Szep}, and the proof for the remaining surfaces $N\in\{N_{2,2}, N_{3,1}, N_{4,0}\}$ is given in the present article (Theorem \ref{thm:aut_curv}). 

The following theorem is the main result of this paper.
\begin{theorem}\label{main:aut_pants} 
If $N$ is a nonorientable surface with $\chi(N)<0$ and 
$\theta\colon\mcg(N)\to\Aut(\pants(N))$ is the natural map, then
\begin{enumerate}
    \item $\theta$ is surjective;
    \item $\ker\theta=\ker\eta$, where  $\eta\colon\mcg(N)\to\Aut(\curv(N))$ is the map from Theorem \ref{main:aut_curv}.
\end{enumerate}
\end{theorem}
As an immediate corollary, we get
\begin{cor}
    If $N=N_{g,b}$ is a nonorientable surface with $\chi(N)<0$, then
\begin{enumerate}
    \item $\Aut(\pants(N))\cong\Aut(\curv(N))$,
    \item $\Aut(\pants(N))\cong\Aut(\curv(N))\cong\mcg(N)$, for $g+b\ge 4$.
\end{enumerate}
\end{cor}
We also prove that if two nonorientable surfaces have isomorphic pants graphs, then they are homeomorphic -- see Theorem \ref{prop:iso:pants}. Finally, we show that it is enough to add 3 types of 2-dimensional cells (two types of triangles and one type of pentagonal cells) to $\pants(N_3)$ to obtain a connected and simply-connected complex $\overline{\pants(N_3)}\supset \pants(N_3)$ -- see Theorem \ref{prop:n30:2dim}.
\subsection{Comparison with the case of an orientable surface.}
Our proof of Theorem \ref{main:aut_pants} has a similar structure as the proof of Theorem \ref{Th:Margalit} in \cite{Mar} however, there are some important differences. 

If $S$ is an oriented surface, then to each edge of $\pants(S)$ there is an associated Farey graph and these graphs are the main tool to control the action of elements of $\Aut(\pants(S))$. If $N$ is nonorientable, then the situation is more complicated, because there are different kinds of subgraphs associated to edges in $\pants(N)$ -- see Proposition \ref{lem:moves_char}. What is more, there are edges in $\pants(N)$ that are not contained in any triangle -- see Example \ref{ex:lonely:edge}.  

Another new phenomena in a nonorientable case is the fact that two pants decompositions of the same surface can have a different number of components (see for example Figure \ref{fig:n4:3pants}). The existence of moves of type 4 (Figure \ref{fig:moves:all}) implies that if $N$ has a pants decomposition with $p\geq 2$ one-sided and $q$ two-sided curves, then $N$ has also a pants decomposition with $p-2$ one-sided and $q+1$ two-sided curves. This significantly complicates the notion of 2-tight paths in $\pants(N)$ (corresponding to the notion of 2-curve small paths in \cite{Mar}) -- see Definition \ref{def:two:tight}.

Finally, in an oriented case, the pants graph $\pants^1(S)$ is embedded in a simply-connected pants complex $\pants(S)$ and we have a short list of types of 2-cells in that complex. It turns out that these 2-dimensional cells play a crucial role in the proof of Theorem \ref{Th:Margalit}. On the other hand, if $N$ is nonorientable, then we only have a connected graph $\pants(N)=\pants^1(N)$ and the full list of types of 2-dimensional cells are yet to be discovered. In fact, in the present paper we discovered three interesting 2-cells: pentagon in $N_3$, tame alternating hexagon in $\pants(N_{1,3})$ and standard heptagon in $\pants(N_{2,2})$ -- see Figures \ref{fig:N30:pent}, \ref{fig:6gon} and \ref{fig7gon}. We expect that this kind of cells will be crucial in constructing the 2-dimensional structure of $\pants(N)$, but we leave this for future considerations.
\subsection{Outline of the paper}
In Section \ref{sec:pants:graph} we review the definition of the pants graph $\pants(N)$ of a nonorientable surface $N=N_{g,b}$. We also identify this graph for surfaces of some small values of $g$ and $b$. In Section \ref{sec:Farey} we identify the subgraphs of $\pants(N)$ corresponding to edges of $\pants(N)$ (Remark \ref{rem:edge:graphs}). Then we classify the triangles in $\pants(N)$ (Proposition \ref{lem:triangles}) and we prove the abstract characterization of moves (Proposition \ref{lem:moves_char}). In Section \ref{sec:Circuits} we study properties of short loops (circuits) in $\pants(N)$ and we introduce the circuit corresponding to 2-cell in $\pants(N_{2,2})$ (standard heptagon). The analysis of short circuits is continued in Section \ref{sec:tame:circ} and we prove that (in most cases) every edge of $\pants(N)$ can be extended to a short circuit with some special properties (Proposition \ref{prop:edge-circuit-new}). This construction is a crucial step towards the proof that there is a well-defined correspondence between $\Aut(\pants(N))$ and $\Aut(\curv(N))$ -- we prove this in Section \ref{sec:aut:pants}, Theorem \ref{prop:Phi:iso}. In the same section, we finish the proof of our main Theorem \ref{main:aut_pants}. In Section \ref{sec:iso:pants} we prove that if two nonorientable surfaces have isomorphic pants graphs, then they are homeomorphic (Theorem \ref{prop:iso:pants}). Section \ref{sec:aut:curves} contains the proof of Theorem \ref{main:aut_curv} if $g+b=4$ (Theorem~\ref{thm:aut_curv}). Finally, in Section \ref{sec:N30} we focus on the case $N=N_3$. We describe the structure of $\pants(N)$ and construct an abstract graph isomorphic to it. Additionally, we  define a simply-connected 2-dimensional complex $\overline{\pants(N)}$, whose
      1-skeleton is $\pants(N)$, and such that $\Aut(\overline{\pants(N)})=\Aut(\pants(N))$ (Theorem \ref{prop:n30:2dim}).


%% file: Preliminaries.tex
\section{Preliminaries} \label{sec:preli}

Let $S$ be a connected compact surface of negative Euler characteristic. Recall that $\chi(S_{g,b})=2-2g-b$, whereas  $\chi(N_{g,b})=2-g-b$. 
A \emph{non-trivial curve} on $S$ is an embedded simple closed curve which does not bound a disc nor a M\"obius band in $S$ and which is not homotopic to a boundary component of $S$. By a curve we always understand a non-trivial curve. A curve is either two- or one-sided depending on whether its regular neighbourhood is an annulus or a M\"obius band respectively. To simplify the notation, we identify a curve with its isotopy class. We denote by $i(\alpha,\beta)$ the geometric intersection number of two curves $\alpha,\beta$. If $i(\alpha,\beta)=0$ and $\alpha\ne\beta$, then we say that $\alpha$ and $\beta$ are \emph{disjoint}.

A (finite) set of pairwise disjoint curves on $S$ is called a \emph{multicurve}. For a multicurve $X$ we denote by $X^+$ (resp. $X^-$) the set of two-sided (resp. one-sided) curves of $X$, and by $|X|$ the number of curves in $X$.

The \emph{curve complex} $\curv(S)$ is a simplicial complex whose $k$-simplices are multicurves of cardinality $k+1$ on $S$. 
The $k$-skeleton of $\curv(S)$ is denoted by $\curv^k(S)$.

If $S=S_{0,4}$ or $S=S_{1,1}$, the definition of $\curv(S)$ is modified by declaring $\alpha,\beta\in\curv^0(S)$ to be adjacent in $\curv(S)$ whenever $i(\alpha,\beta)=2$ or $i(\alpha,\beta)=1$ respectively. The resulting connected graph is isomorphic to \emph{the Farey graph} -- see Section 3 of \cite{Minsky1996AGA}. For the description of $\curv(N)$ for nonorientable surfaces of small complexity see \cite{Sch}.

For a multicurve $X$ on $S$ we denote by $S_X$ the surface obtained by cutting $S$ along $X$. A component of $S_X$ is \emph{trivial} if it is homeomorphic to $S_{0,3}$. A \emph{pants decomposition} of $S$ is a multicurve $X$ such that all components of $S_X$ are trivial. Note that a pants decomposition is a maximal multicurve. 

Any pants decomposition of an orientable surface $S_{g,b}$ consists of $3g+b-3$ curves. The number of curves in a pants decomposition of a nonorientable surface $N_{g,b}$ varies. For any integer $m$ such that $0\le m\le g$ and $m\equiv g\pmod{2}$ there is a pants decoposition of $N_{g,b}$ of cardinality 
\[\frac{3g+m}{2}+b-3\] containing exactly $m$ one-sided curves -- see Figure \ref{fig:n4:3pants} (the shaded disks represent crosscaps, hence their interiors are to be removed and then the antipodal points on each boundary component are to be identified). 
    \begin{figure}[h]
\begin{center}
\includegraphics[width=0.99\customwidth]{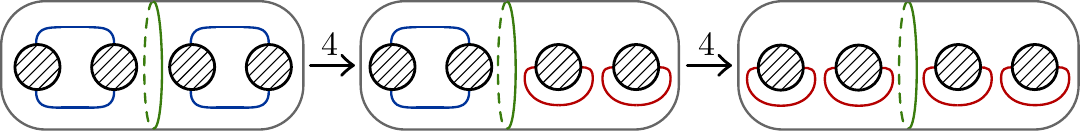} 
\caption{Pants decompositions of $N_4$ with different cardinalities.}\label{fig:n4:3pants} %
\end{center}
\end{figure}

%% file: Pants.tex
\section{The pants graph}\label{sec:pants:graph}
\subsection{Definition and connectivity.}
Let $N$ be a nonorienatble surface with $\chi(N)<0$. We denote by $\pants^0(N)$ the set of all pants decompositions of $N$.
The \emph{pants graph} of $N$, denoted by $\pants(N)$ is a graph whose vertex set is $\pants^0(N)$ and whose edges correspond to elementary moves between pants decompositions. 
An \emph{elementary move} is an operation of replacing one curve $\alpha$ of a pants decomposition $X$ either by a curve $\beta$, or by a pair of disjoint one-sided curves $\{\beta_1,\beta_2\}$ to obtain a different pants decomposition $Y$. There are four types of elementary move depending on the nontrvial component of
$N_{X\setminus\{\alpha\}}$ called the \emph{support} of the move. The intersection numbers of the involved curves are given in the following table (see also Figure~\ref{fig:moves:all}).

\medskip
\begin{center}
  \begin{tabular}{|c|c|c|c|c|}
\hline
    Type & Support & Notation & Curves & Intersection \\
\hline
  \rule{0pt}{15pt}  $1$ & $S_{1,1}$ & ${\alpha}{\mvI}{\beta}$ & $\alpha\in X^+,\ \beta\in Y^+$  & $i(\alpha,\beta)=1$ \\
\hline
 \rule{0pt}{15pt}  $2$ & $S_{0,4}$ & ${\alpha}{\mvII}{\beta}$ & $\alpha\in X^+,\ \beta\in Y^+$ & $i(\alpha,\beta)=2$ \\
\hline
 \rule{0pt}{15pt}  $3$ & $N_{1,2}$ & ${\alpha}{\mvIII}{\beta}$ & $\alpha\in X^-,\ \beta\in Y^-$ & $i(\alpha,\beta)=1$ \\
\hline
 \rule{0pt}{15pt}  $4$ & $N_{2,1}$ & ${\alpha}{\mvIV}{\{\beta_1,\beta_2\}}$ & $\alpha\in X^+,\ \{\beta_1,\beta_2\}\subseteq Y^-$ & $i(\alpha,\beta_1)=i(\alpha,\beta_2)=1$\\
\hline
\end{tabular}  
\end{center}

\begin{figure}[h]
\begin{center}
\includegraphics[width=0.99\customwidth]{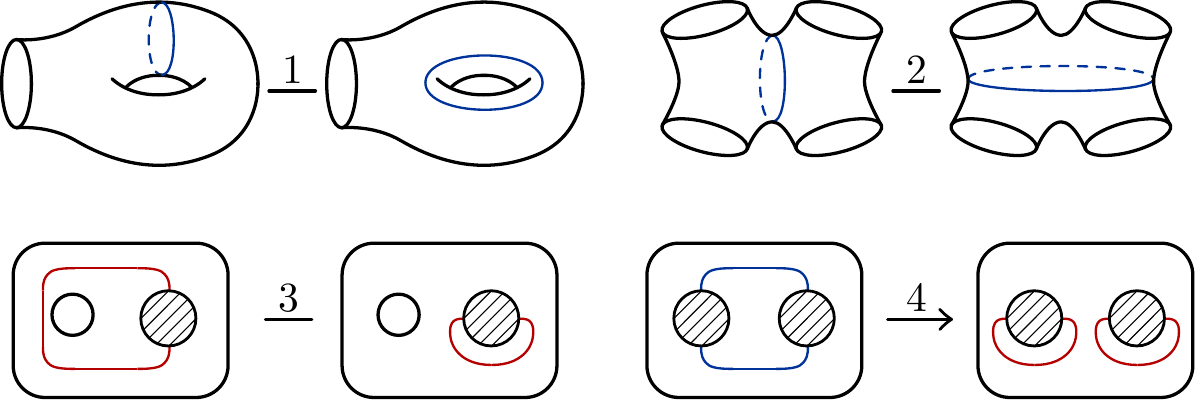}
\caption{The four elementary moves between pants decompositions.}\label{fig:moves:all} %
\end{center}
\end{figure}

\medskip
A move a type 4 has a natural direction towards the pants decomposition of bigger cardinality, which we indicate by an arrow. For the sake of symmetry, the inverse of such a move, replacing a pair of disjoint one-sided curves by a single two-sided curve, will also be considered a move of type 4.

The pants graph of an orientable surface is defined analogously, using only moves of type 1 and 2 coming from the seminal paper 
\cite{HatcherThurston}. Moves of type 3 and 4, supported on nonorientable surfaces and involving one-sided curves, were introduced by Papadopoulos and Penner \cite{PP} who proved the following fundamental result.
\begin{theorem}[\cite{PP}]\label{pants_connected}
For every nonorientable surface $N$ with $\chi(N)<0$ the pants graph $\pants(N)$ is connected.
\end{theorem}

\subsection{Notation for edges.}
  We denote by $XY$ an  edge of $\pants(N)$ (elementary move) joining vertices $X$ and $Y$. We also write 
    \[X{\mvi}Y\quad\textrm{or}\quad \fullMv{XY}{\alpha}{\mvi}{\beta}\quad\textrm{or}\quad {\alpha}{\mvi}{\beta}\]
    to denote a move of type $i$ replacing $\alpha\in X$ by $\beta\in Y$. The type of the move may be dropped if it is arbitrary or unknown. A move a type 4, changing the number of curves in a pants decomposition, has a natural direction towards the vertex of bigger cardinality, which we indicate by an arrow
    \[X{\mvIV}Y\quad\textrm{or}\quad \fullMv{XY}{\alpha}{\mvIV}\{\beta_1,\beta_2\}\quad\textrm{or}\quad {\alpha}{\mvIV}\{\beta_1,\beta_2\}.\]
    Unless explicitly stated otherwise, letters of the Greek alphabet denote curves. An underlined Greek letter, e.g. $\underline{\alpha}$, denotes either a curve, or a multicurve consisting of two one-sided curves. For example
    \[\fullMv{XY}{\alpha}{\mv}{\underline{\alpha'}}\]
    denotes a move replacing a curve $\alpha$ by $\underline{\alpha'}$, which is either a single curve, or a pair of  one-sided curves. In the latter case $XY$ is a move of type $4$ directed from $X$ to $Y$.

\subsection{Pants graph of $S_{0,4}$, $S_{1,1}$, $N_{1,2}$ and $N_{2,1}$.}
The graphs $\pants(S_{0,4})$ and $\pants(S_{1,1})$ are the same as the curve complexes of these surfaces and are isomorphic to the Farey graph -- see Section 3 of \cite{Minsky1996AGA}.
\begin{prop}[Pants graph of $N_{1,2}$, \cite{PP}]\label{pantsN12}
The graph $\pants(N_{1,2})$ consists of two vertices, 
corresponding to two one-sided curves, joined by an edge of type 3.
\end{prop}

\begin{defi}[Infinite fan]
The \emph{infinite fan} $F_\infty$ is the graph with vertex set $\{v\}\cup\{v_i\}_{i\in\Z}$ and edge set set
$\{vv_i,v_iv_{i+1}\}_{i\in\Z}$. The vertex $v$ is called the \emph{centre} of $F_\infty$.
\end{defi}

\begin{prop}[Pants graph of $N_{2,1}$, \cite{PP}]\label{pantsN21}
The graph $\pants(N_{2,1})$ is isomorphic to $F_\infty$. Every isomorphism $F_\infty\to\pants(N_{2,1})$ maps the centre of $F_\infty$ to the pants decomposition consisting of the unique non-separating two-sided curve in $N_{2,1}$. 
\end{prop}
\begin{proof} We only show how to define an isomorphism $F_\infty\to\pants(N_{2,1})$, see \cite{PP} for  details.
    Let $\alpha$ be the unique non-separating two-sided curve in $N=N_{2,1}$ and let $T_\alpha\in\mcg(N)$ be a Dehn twist about $\alpha$. We have $\{\alpha\}\in\pants^0(N)$. Choose a one-sided curve $\beta$ in $N$. Since $N_{\beta}\cong N_{1,2}$, by Proposition \ref{pantsN12} there are exactly two pants decompositions of $N$ containing $\beta$, namely
    $\{\beta,T_\alpha(\beta)\}$ and $\{\beta,T_\alpha^{-1}(\beta)\}$, joined in $\pants(N)$ by an edge of type 3. Each of these vertices is also connected to $\{\alpha\}$ by an edge of type 4. The map $\phi\colon F_\infty\to\pants(N)$ defined by
    \[\phi(v)=\{\alpha\},\quad\phi(v_i)=\{T_\alpha^i(\beta),T_\alpha^{i+1}(\beta)\}, i\in\Z\]
    is an isomorphisms of graphs.
\end{proof} 

%% file: F_graphs.tex
\section{Farey and fan subgraphs}\label{sec:Farey}
\begin{defi}
For $X\in\pants(N)$ and $\alpha\in X$ we denote by $F(X,\alpha)$ the subgraph of $\pants(N)$ spanned by all the pants decompositions containing $X\setminus\{\alpha\}$.    
\end{defi}
\begin{rem}\label{rem:FAal:char}
For $X\in\pants(N)$ and $\alpha\in X$ let $M$ be the non-trivial (i.e. $M\not\cong S_{0,3}$) component of $N_{X\setminus\{\alpha\}}$. Then $F(X,\alpha)\cong\pants(M)$ and if $\alpha\in X^+$, then $M\in\{S_{0,4}, S_{1,1}, N_{2,1}\}$, whereas if $\alpha\in X^-$, then $M\cong N_{1,2}$. Therefore, we have the following characterization of the subgraphs $F(X,\alpha)$.
\begin{enumerate}
    \item If $\alpha\in X^+$ and $M\cong S_{0,4}$ or $M\cong S_{1,1}$, then $F(X,\alpha)$ is isomorphic to the Farey graph.
    \item If $\alpha\in X^+$ and $M\cong N_{2,1}$, then $F(X,\alpha)$ is isomorphic to the infinite fan $F_\infty$ by Proposition \ref{pantsN21}. Note that in this case $X$ is the unique vertex of infinite degree in $F(X,\alpha)$ and we call it the \emph{centre} of $F(X,\alpha)$.
    \item If $\alpha\in X^-$, then $F(X,\alpha)$ consist of two vertices joined by an edge of type 3 replacing $\alpha$ by a uniquely determined $\alpha'$ by Proposition \ref{pantsN12}.
\end{enumerate}   
\end{rem}

\begin{defi} Let $X\in\pants(N)$ and $\alpha\in X^+$. We say that $\alpha$ is \emph{special} in $X$ if  the non-trivial component of $N_{X\setminus\{\alpha\}}$ is a one-holed Klein bottle $N_{2,1}$.
\end{defi}
\begin{rem} \label{rem:special:in:fan}
     By Remark \ref{rem:FAal:char},  $F(X,\alpha)\cong F_{\infty}$  if and only if $\alpha$ is special in $X$. Moreover, if $\alpha$ is special in $X$, then $X$ is the centre of $F(X,\alpha)$.
\end{rem}
\begin{rem} \label{rem:special:in:mvIV}
By the previous remark, if 
\[\fullMv{XY}{\alpha}{\mvIV}{\{\beta,\beta'\}}\]
is a move of type 4, then $\alpha$ is special in $X$, $\alpha\not\in Y$ and $XY$ is an edge of $F(X,\alpha)\cong F_{\infty}$. Moreover, $X$ is of infinite degree in $F(X,\alpha)$.
\end{rem}
\begin{defi} \label{defi:XY:graph}
    Let $XY$ be an edge of $\pants(N)$. We denote by $F(XY)$ the subgraph of $\pants(N)$ spanned by all pants decompositions containing $X\cap Y$.    
\end{defi}
%
\begin{rem}\label{rem:edge:graphs}
If $\fullMv{XY}{\alpha}{\mv }{\underline{\alpha'}}$
then $F(XY)=F(X,\alpha)$. 
    By Remarks \ref{rem:FAal:char}--\ref{rem:special:in:mvIV}, $F(XY)$ is isomorphic 
    \begin{enumerate}
    \item to the Farey graph if $XY$ is of type 1 or 2;
    \item to the infinite fan $F_{\infty}$ if $XY$ is of type 4;
    \item to the single edge graph if $XY$ is of type 3.
    \end{enumerate}
\end{rem}
\begin{lemma}\label{lem:special}
If $XY$ is a move of type 1, 2 or 3 and $\alpha$ is special in $X$, then $\alpha\in Y$. 
\end{lemma}
\begin{proof}
Suppose $\alpha\notin Y$. Then $XY$ must replace $\alpha$ by some $\beta\in Y$. By definition $\alpha\in X^+$, hence $\beta\in Y^+$ and $i(\alpha,\beta)>0$. 
Let $K$ be the one-holed Klein bottle component of $N_{X\setminus\{\alpha\}}=N_{Y\setminus\{\beta\}}$.
Since $\alpha$ is the unique two-sided curve in $K$ we have $\beta=\alpha$ which is a contradiction.
\end{proof}
\begin{prop}[Classification of triangles]\label{lem:triangles} Every triangle in $\pants(N)$ is of one of the  two forms. Either
\[\xymatrix@R=0.1pc{
&&\alpha'\ar@{-}[dd]^i\\
(1)&\alpha\ar@{-}[ru]^i\ar@{-}[rd]_i&\\
&&\alpha''
}\]
 where $i\in\{1,2\}$, or 
 \[\xymatrix@R=0.1pc{
&&\{\beta_1,\beta_2\}\ar@{-}[dd]^3\\
(2)&\alpha\ar[ru]^(.35)4\ar[rd]_(.35)4&\\  %
&&\{\beta_1',\beta_2\}
}\]
 A triangle of the form (1) (resp. (2)) will be called a \emph{Farey triangle} (resp. \emph{fan triangle}).
\end{prop}
\begin{proof} Let $X,Y,Z\in\pants(N)$ be vertices of a triangle.

 Suppose first that $|X|=|Y|=|Z|=n$ and $XY$ is a move replacing $\alpha\in X$ by $\alpha'\in Y$. We have $i(\alpha,\alpha')>0$ so $\alpha\notin Z$ or 
 $\alpha'\notin Z$. But $|X\cap Z|=|Y\cap Z|=n-1$ which implies $X\cap Z=Y\cap Z=X\setminus\{\alpha\}$. We see that $XYZ$ is of the form (1) but we still need to show that $i\ne 3$. Suppose that $\alpha$ is one-sided. The non-trivial component of $N_{X\setminus\{\alpha\}}$ is homeomorphic to $N_{1,2}$. This component contains $\alpha,\alpha'$ and $\alpha''$, which is a contradiction, because there are only two isotopy classes of nontrivial curves in $N_{1,2}$ -- see Proposition \ref{pantsN12}.  

 Now suppose $n=|X|<|Y|$ and $XY$ is a move of type 4 replacing $\alpha\in X^+$ by $\beta_1,\beta_2\in Y^-$. We have $n\le|Z|\le n+1$ so we consider two cases. First suppose $|Z|=n$ so that $XZ$ is a move of type 1, 2 or 3, and $ZY$ is a move of type 4. Since $\alpha$ is special in $X$, by Lemma~\ref{lem:special}, $\alpha\in Z$. We see that $ZY$ replaces $\alpha$ by $\beta_1,\beta_2$ so $ZY=XY$ and $X=Z$, a contradiction. Therefore, $|Z|=n+1$ and $XZ$ is a move of type 4. Now $|Z\cap Y|=n$ and we can assume that $\beta_2\in Z$. We have $\alpha\notin Z$ because $i(\beta_2,\alpha)=1$ so $XZ$ replaces $\alpha$ by $\beta_1',\beta_2$ for some one-sided $\beta'_1$, $\beta'_1\ne\beta_1$. It follows that $YZ$ is a move of type 3 replacing $\beta_1$ by $\beta'_1$.
\end{proof}
\begin{cor}\label{cor:2:triangles}
Let $XYZ$ be a triangle in $\pants(N)$ and assume that $YZ$ is of type 3 if $XYZ$ is a fan triangle. Then
\begin{enumerate}
    \item $XYZ\subset F(X,\alpha)=F(XY)$ for a unique $\alpha\in X^+$;
    \item if $XYZ'$ is another triangle in $\pants(N)$ and $YZ'$ is of type 3 if $XYZ'$ is a fan triangle, then
    \[XYZ'\subset F(X,\alpha)=F(XY),\]
    where $\alpha$ is as in (1).
\end{enumerate}
\end{cor}
\begin{proof}
    If $\fullMv{XY}{\alpha}{\mv }{\multi{\alpha}'}$, then Proposition \ref{lem:triangles} implies that
    \begin{align*}
    XYZ&\subset F(X,\alpha)\\
    XYZ'&\subset F(X,\alpha).
    \end{align*}
    Combine this with Definition \ref{defi:XY:graph} to get the statement.
\end{proof}
\begin{lemma}\label{lem:move3}
Suppose that $XY$ is an edge of type $3$ in $\pants(N)$.
\begin{enumerate}  
    \item If $N\ne N_3$, then $XY$ belongs to at most one triangle.
    \item If $N=N_3$, then $XY$ either does not belong to any triangle, or it belongs to exactly two different triangles. It belongs to two triangles if and only if $X$ and $Y$ are vertices of degree $6$ in $\pants(N)$.
\end{enumerate}
\end{lemma}
\begin{proof}
Let $\fullMv{XY}{\beta_1}{\mvIII}{\beta'_1}$. 
By Lemma \ref{lem:triangles}, an edge of type 3 can belong only to a fan triangle, and $XY$ belongs to such a triangle if and only if 
there is a multicurve $\gamma\subset(X\setminus\beta_1)=(Y\setminus\beta'_1)$ such that the only non-trivial component of $N_\gamma$ is a one-holed Klein bottle $K$ containing $\beta_1$ and $\beta'_1$. Such $\gamma$ uniquely determines a fan triangle $XYZ$ of the form
\[\fullMv{ZX}{\alpha}{\mvIV}{\{\beta_1,\beta_2\}},\quad \fullMv{XY}{\beta_1}{\mvIII}{\beta'_1},\quad \fullMv{YZ}{\{\beta'_1,\beta_2\}}{\mvIVL}{\alpha}.\]
 Indeed, $\alpha$ is the unique two-sided curve in $K$ and $\beta_2\in X^-\cap Y^-$ is the unique one-sided curve in $K$ disjoint form 
  $\beta_1\cup\beta'_1$ -- see Figure \ref{fig:n21:fan:curves}.
  \begin{figure}[h]
\begin{center}
\includegraphics[width=0.45\customwidth]{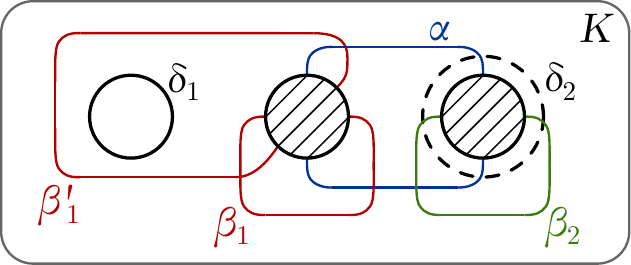}
\caption{Curves of the fan triangle in $N_{2,1}$.}\label{fig:n21:fan:curves} %
\end{center}
\end{figure}

Suppose that such $\gamma$ exists. The regular neighbourhood of $\beta_1\cup \beta'_1$ is a projective plane $N_{1,2}$ with two boundary components $\delta_1,\delta_2$ and the existence of $\gamma$ implies that at least one of these two boundary components, say $\delta_2$, cuts a projective plane from $N$ (Figure \ref{fig:n21:fan:curves}). Therefore, if $N\ne N_3$, then $\gamma$ is unique: $\gamma=X\setminus\{\beta_1,\beta_2\}$.

If $N=N_3$, then $\gamma$ is a single one-sided curve disjoint from $\beta_1\cup\beta'_1$, and there are two such curves (Figure \ref{fig:n3:two:bottles}).
\begin{figure}[h]
\begin{center}
\includegraphics[width=0.45\customwidth]{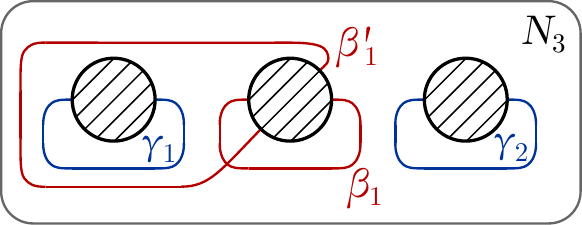}
\caption{One-sided curves in $N_3$.}\label{fig:n3:two:bottles} %
\end{center}
\end{figure}
In this case each of the vertices $X$ and $Y$ is a pants decomposition consisting of three one-sided curves. Such a vertex has degree $6$, as it is incident to three edges of type 3 and three edges of type 4 (Figure \ref{fig:n3:star}).
\begin{figure}[h]
\begin{center}
\includegraphics[width=0.95\customwidth]{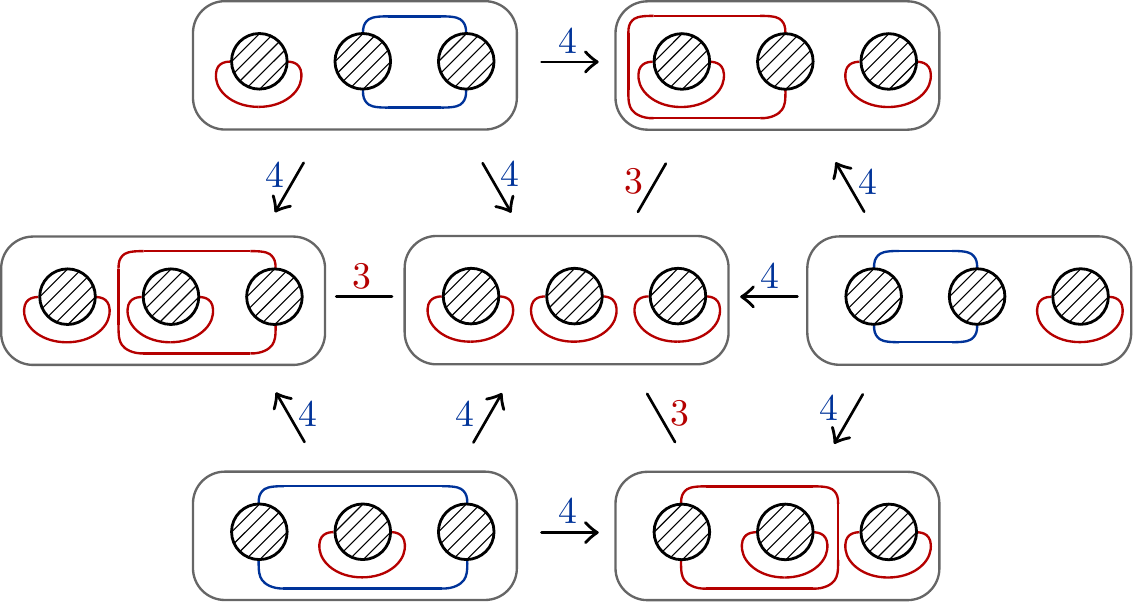}
\caption{The vertex of degree 6 in $\pants(N_3)$.}\label{fig:n3:star} %
\end{center}
\end{figure}
\end{proof}

\begin{example}\label{ex:lonely:edge}
Let 
\[\fullMv{XY}{\{\alpha,\beta\}}{\mvIII}{\{\alpha,\beta'\}}\]
be an edge of type 3 in $\pants(N_3)$, where $\alpha,\alpha',\beta$ are as in Figure~\ref{figNonfan3}.
\begin{figure}[h]
\begin{center}
\includegraphics[width=0.67\customwidth]{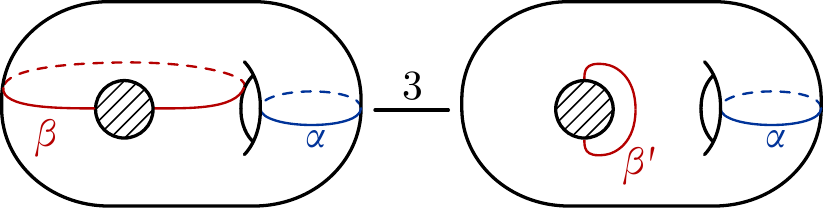} 
\caption{An edge of type 3 in $\pants(N_3)$ that is not contained in any triangle.}\label{figNonfan3} %
\end{center}
\end{figure}
Since the complement of $\beta$ in $N_3$ is orientable, this edge can not be contained in a fan triangle. Moreover, this edge cannot be contained in a Farey triangle, because $\beta$ is one-sided. By Proposition \ref{lem:triangles}, $XY$ is not contained in any triangle.
\end{example}
\begin{defi} \label{def:subgraph}
Let $F$ be a subgraph of $\pants(N)$. We say that 
\begin{enumerate}
    \item $F$ is a \emph{Farey subgraph} if $F$ is isomorphic to the Farey graph;
    \item $F$ is a \emph{fan subgraph} if $F$ is isomorphic to the infinite fan $F_\infty$ and $F$ is not contained in any Farey subgraph;
    \item $F$ is a \emph{thin edge} if $F$ consists of a single edge, $F$ is not contained in any Farey subgraph and if $F$ is contained in a fan subgraph $F'$, then both vertices of $F$ are of finite degree in $F'$.
\end{enumerate}
\end{defi}
\begin{rem}
    Note that every Farey subgraph contains subgraphs that are isomorphic to $F_\infty$. However, by the above definition, they are not fan subgraphs of $\pants(N)$.
\end{rem}
\begin{lemma}\label{edge3:inFarey}
    Let $F$ be an edge of type 3 in $\pants(N)$. Then $F$ is a thin edge.
\end{lemma}
\begin{proof}
    Every edge in a Farey subgraph $F'$ belongs to two different triangles in $F'$ and every vertex of $F'$ has infinite degree in $F'$. Therefore, by Lemma \ref{lem:move3}, $F$ can not be contained in $F'$.
    
    If $F$ is contained in a fan subgraph $F'$ and one of the vertices of $F$ is of infinite degree in $F'$, then $F$ belongs to two different triangles in $F'$. This contradicts Lemma \ref{lem:move3}.
\end{proof}
\begin{lemma}\label{edge4:inFarey}
If 
\[\fullMv{XY}{\alpha}{\mvIV}{\{\beta_1,\beta_2\}}\] 
is an edge of type 4 in $\pants(N)$, then the fan graph $F(XY)=F(X,\alpha)$ is not contained in any Farey subgraph, so it is a fan subgraph of $P(N)$.
\end{lemma}
\begin{proof}
    Recall that $F(XY)=F(X,\alpha)\cong \pants(N_{2,1})$ is a fan subgraph by Remarks \ref{rem:FAal:char} and \ref{rem:special:in:mvIV}. Moreover, by the proof of Proposition \ref{pantsN21}, $\pants(N_{2,1})$ contains edges of type 3. By Lemma~\ref{edge3:inFarey}, $F(XY)$ can not be contained in any Farey subgraph.
\end{proof}
\begin{defi}\label{defi:f:graph}
A \emph{marked F-graph} is a pair $(F,X)$, where $X$ is a vertex of $F$ and $F$ is either
\begin{enumerate}
    \item a Farey subgraph of $\pants(N)$, or
    \item a fan subgraph of $\pants(N)$, and $X$ is the centre of $F$, or else 
    \item a single thin edge in $\pants(N)$.
\end{enumerate}
\end{defi}
\begin{rem}\label{rem:edge:to:Fg}
If we combine Remark \ref{rem:edge:graphs} and Lemmas \ref{edge3:inFarey} and \ref{edge4:inFarey}, we see that to every oriented 
edge $XY$ of $\pants(N)$ with $|X|\leq |Y|$, there is an associated marked F-graph $(F(XY),X)$. 
\end{rem}
 The following proposition shows that the above correspondence is surjective.
\begin{prop}\label{lem:Fsubgraph}
If $(F,X)$ is a marked F-graph, then $F=F(X,\alpha)$ for a unique $\alpha\in X$. Moreover, 
\begin{enumerate}
    \item $\alpha\in X^+$ if and only if $F$ is a Farey or a fan subgraph; 
    \item $\alpha$ is special in $X$ if and only if $F$ is a fan subgraph.
\end{enumerate}
\end{prop}
\begin{proof} 
Suppose first that $F$ is a Farey of fan subgraph of $\pants(N)$ and let $T$ be a triangle of $F$ with vertices $X,Y,Z$, where $X$ is of infinite degree in $F$ (see Definition \ref{defi:f:graph}). By Proposition~\ref{lem:triangles}, $T$ is either a Farey triangle or a fan triangle. In the latter case, since $X$ is of infinite degree in $F$, each of the edges $XY$ and $XZ$ belongs to two different triangles in $F$, hence Lemma \ref{lem:move3} implies that $XY$ and $XZ$ are moves of type 4. In either case $T$ satisfies the assumptions of Corollary \ref{cor:2:triangles}, so $T\subset F(X,\alpha)$ for a unique $\alpha\in X^+$. 

Now, if $T'$ is any triangle of $F$, then $T'$ and $T$ are connected by a sequence of triangles such that every two consecutive triangles in that sequence satisfy the assumptions of part (2) of Corollary \ref{cor:2:triangles}. It follows that all these triangles are contained in $F(X,\alpha)$ and consequently $F\subseteq F(X,\alpha)$. Since every injective simplicial map $F\to F$ is an isomorphism, $F=F(X,\alpha)$.


Assume now that $F$ is a thin edge:
\[\fullMv{XY}{\alpha}{\mv}{\multi{\alpha}'}\]
If $XY$ is an edge of type 1 or 2, then $F$ is contained in a Farey graph $F(XY)=F(X,\alpha)$. This contradicts the definition of a thin edge. Similarly, if $XY$ is of type 4, then Remark \ref{rem:special:in:mvIV} and Lemma \ref{edge4:inFarey} imply that $X$ is of infinite degree in a fan subgraph $F(XY)=F(X,\alpha)$. This contradicts the definition of a thin edge. Hence, $F$ is of type 3, $F=F(X,\alpha)$ (see Remark \ref{rem:FAal:char}) and $\alpha$ is the unique element of $X$ such that $i(\alpha,\alpha')>0$. 

Statements (1) and (2) follow from Remarks \ref{rem:FAal:char} and \ref{rem:special:in:fan}.


\end{proof}
The construction of $F(X,\alpha)$ in the proof of the above proposition implies the following corollary.
\begin{cor}\label{cor:edge:in:F}
    If $(F,X)$ is a marked F-graph and 
    \[\fullMv{XY}{\alpha}{\mv}{\multi{\alpha}'}\]
    is an edge of $F$, then $F=F(X,\alpha)=F(XY)$.
\end{cor}
\begin{cor}\label{lem:trian:in:Fgraphs}
    If $F$ is a Farey (resp. fan) subgraph of $\pants(N)$, then every triangle of $F$ is a Farey (resp. fan) triangle.
\end{cor}
\begin{proof}
Let $F= F(X,\alpha)$ be as in Proposition \ref{lem:Fsubgraph}. By Remark \ref{rem:FAal:char}, all triangles of $F(X,\alpha)$ are either fan or Farey depending on whether $\alpha\in X^+$ is special or not.
\end{proof}
\begin{cor}\label{A:inv:F:graph}
    If $N, N'$ are nonorientable surfaces of negative Euler characteristic and $A\colon\pants(N)\to\pants(N')$ is an isomorphism, then for any edge $XY$ of $\pants(N)$
    \[A(F(XY))=F(A(X)A(Y)).\]
\end{cor}
\begin{proof} Assume $|X|\leq |Y|$. 
    By Remark \ref{rem:edge:to:Fg}, $(F(XY),X)$ is a marked F-graph, so \[(A(F(XY)),A(X))\] is also a marked F-graph (see Definitions \ref{def:subgraph} and \ref{defi:f:graph}). Moreover, $A(X)A(Y)$ is and edge of $A(F(XY))$, so Corollary \ref{cor:edge:in:F} implies that
    \[A(F(XY))=F(A(X)A(Y)).\]
\end{proof}

\begin{prop}[Characterisation of moves]\label{lem:moves_char}
Let $XY$ be an edge of $\pants(N)$. 
\begin{enumerate}
    \item $XY$ is a move of type 1 or 2 if and only if $F(XY)$ is a Farey subgraph;
    \item $XY$ is a move of type 4 directed from $X$ to $Y$ if and only if $F(XY)$ is a fan subgraph with centre $X$; 
    \item $XY$ is a move of type 3 if and only if $F(XY)$ is a thin edge.
\end{enumerate}
\end{prop}
\begin{proof} Assume $|X|\le |Y|$. 
    By Remark \ref{rem:edge:to:Fg}, $(F(XY),X)$ is a marked F-graph and its type is uniquely determined by the type of $XY$ by Remark \ref{rem:edge:graphs}.
    
    Conversely, if $F(XY)$ is a Farey subgraph, then Corollary \ref{lem:trian:in:Fgraphs} implies that $XY$ is of type 1 or 2. If $F(XY)$ is a fan subgraph with centre $X$, then Corollary \ref{lem:trian:in:Fgraphs} implies that $XY$ is of type 3 or 4. But we also know that $X$ is of infinite degree in $F(XY)$, so Lemma \ref{edge3:inFarey} implies that $XY$ is an edge of type 4 and $|X|<|Y|$. Finally, if $F(XY)$ is a thin edge, then $XY$ can not be of types 1, 2 or 4 by (1) and (2).
\end{proof}
\begin{rem}
    Note that there are two different types of edges of type 3 -- some of them are contained in some fan subgraph of $\pants(N)$, but some of them are not -- see Example \ref{ex:lonely:edge}.
\end{rem}
\begin{cor}\label{cor:type_preserv}
    If $N$, $N'$ are nonorientable surfaces of negative Euler characteristic and $A\colon\pants(N)\to\pants(N')$ is an isomorphism, then $A$ maps an edge of type $3$ or $4$ to an edge of the same type. Furthermore, $A$ also preserves directions of edges of type 4. 
\end{cor}
\begin{proof}
By Corollary \ref{A:inv:F:graph}, if $XY$ is an edge of type 3 or 4 in $\pants(N)$, then 
\[F(A(X)A(Y))=A(F(XY))\cong F(XY),\]
so $A(X)A(Y)$ and $XY$ are of the same type by Proposition \ref{lem:moves_char}. Furthermore, if $|X|<|Y|$ then $X$ is the centre of 
$F(XY)$,  so $A(X)$ is the centre of 
$A(F(XY))$ and  $|A(X)|<|A(Y)|$.
\end{proof}

%% file: Circuits_new.tex
\section{Small circuits}\label{sec:Circuits}
\begin{defi}
    A \emph{circuit} (resp. a \emph{path}) is a subgraph of $\pants(N)$ homeomorphic to a circle (resp. an interval). 
    The \emph{length} of a circuit (resp. a path) is the number of its edges.
\end{defi}
We use the names \emph{triangle, quadrangle, pentagon, hexagon, heptagon} for 
circuits of length at most $7$.
\begin{defi}\label{def:alter}
A path or a circuit 
 is \emph{alternating} 
 if no two consecutive edges (i.e. edges incident to a common vertex) are contained in one Farey or fan subgraph.
\end{defi}
\begin{rem}\label{rem:alter:inv}
Alternating paths (resp. circuits) are mapped to alternating paths (resp. circuits) by automorphisms of $\pants(N)$.
\end{rem}
\begin{lemma}\label{lem:alter:char}
\begin{enumerate}
    \item Every alternating path of length 3 is of the form
\begin{equation}\label{eq:alter}
    \{\multi{\alpha_1},\multi{\alpha_2}\}\mv\{\multi{\alpha_1}',\multi{\alpha_2}\}\mv\{\multi{\alpha_1}',\multi{\alpha_2}'\}.
\end{equation}
\item If $P$ is a path of the form \eqref{eq:alter}, then either $P$ is alternating, or $P$ is contained in a fan subgraph and both edges of $P$ are of type 3.
\end{enumerate}
\end{lemma}
Recall that according to our convention, each of $\multi{\alpha_1},\multi{\alpha'_1},\multi{\alpha_2},\multi{\alpha'_2}$ may be a pair of one-sided curves corresponding to a move of type 4.
\begin{proof}
Let $P$ be a path $X\mv Y\mv Z$.

(1) Suppose that $P$ is alternating, $XY\colon\multi{\alpha_1}\mv\multi{\alpha_1}'$, and $\multi{\alpha}_1'\notin Z$.
If $\multi{\alpha_1}'=\alpha_1'$ is a single curve, then the path is contained in $F(Y,\alpha'_1)$, and since it has length 3, $\alpha'_1$ is two-sided and $F(Y,\alpha'_1)$ is either a Farey subgraph or a fan subgraph by Remark \ref{rem:FAal:char}. This is a contradiction, because we assumed that the path is alternating.
If $\multi{\alpha_1}'=\{\beta_1,\beta_2\}$ is a multicurve, then $YZ$ is an edge of type 3 replacing one of the curves 
$\beta_1$ or $\beta_2$. It follows that the path is contained in the fan subgraph $F(X,\alpha_1)$, which is also a contradiction.
Therefore $\multi{\alpha_1}'\in Z$, which implies that the path is of the form \eqref{eq:alter}. 

(2) Suppose that $P$ is of the form \eqref{eq:alter} and $P\subset F$, where $F$ is either a Farey or a fan subgraph. If $Y$ has infinite degree in $F$, then $(F,Y)$ is a marked $F$-graph and by Lemma~\ref{lem:Fsubgraph} we have $F=F(Y,\alpha)$ for a unique 
$\alpha\in Y$. Since $X,Z\in F$ we have $\alpha'_1=\alpha=\alpha_2$, which is a contradiction. Thus $Y$ has finite degree in $F$, which means that $F$ is a fan subgraph and both edges of $P$ are of type 3. 
\end{proof}


\begin{defi}\label{def:two:tight}
A subgraph $\Gamma$ (path or circuit)  of $\pants(N)$ is    
\emph{2-tight} if there exists $X_0\in\Gamma^0$ such that 
\[\left|\bigcap_{X\in\Gamma^0}X\right|\geq |X_0|-2.\]
\end{defi}
\begin{rem}
    Note that a subpath of a 2-tight path may not be 2-tight. For example, the following circuit of length 4
    \[\xymatrix{ 
    \{\alpha,\beta\}\ar[r]^(.45)4\ar[d]_4&\{\alpha,\beta_1,\beta_2\}\ar[d]^4\\
    \{\alpha_1,\alpha_2,\beta\}\ar[r]_(.45)4&\{\alpha_1,\alpha_2,\beta_1,\beta_2\}
    }\]
    is 2-tight, but its subpath
    \[\{\alpha_1,\alpha_2,\beta\}\mvIV \{\alpha_1,\alpha_2,\beta_1,\beta_2\}\mvIVL \{\alpha,\beta_1,\beta_2\}\]
    is not 2-tight.
\end{rem}

\begin{defi}
 A vertex $X$ of a subgraph $\Gamma$ of $\pants(N)$ is \emph{minimal} in $\Gamma$ if $|X|$ is minimal among all vertices of  $\Gamma$. 
\end{defi}

\begin{lemma}\label{lem:minimal}
    Automorphisms of $\pants(N)$ map minimal vertices of a subgraph $\Gamma$ to minimal vertices of the image of $\Gamma$.
\end{lemma}
\begin{proof}
    A vertex $X$ is minimal in  $\Gamma$ if and only if for every path from $X$ to any other vertex $Y$ of $\Gamma$, the number of edges of type 4 directed from $X$ to $Y$ is greater than the number of edges of type 4 directed from $Y$ to $X$.
    By Corollary \ref{cor:type_preserv} this property is preserved by automorphisms of $\pants(N)$. 
\end{proof}

\begin{lemma}\label{alter_tight_path}
Suppose that $P$ is an alternating and 2-tight path
\[W\mv X\mv Y\mv Z,\] 
such that $X$ is minimal in $P$. Then  $P$ is of the form
\begin{equation}\label{eq:2tight_alter}
    \{\multi{\alpha_1}',\alpha_2\}\mv\{\alpha_1,\alpha_2\}\mv\{\alpha_1,\multi{\alpha_2}'\}\mv\{\multi{\alpha_1}'',\multi{\alpha_2}'\}.
\end{equation}
\end{lemma}
\begin{proof}
 By Lemma \ref{lem:alter:char}, the path $W\mv X\mv Y$ is of the form 
 \[\{\multi{\alpha_1}',\alpha_2\}\mv\{\alpha_1,\alpha_2\}\mv\{\alpha_1,\multi{\alpha_2}'\},\]
 where $\alpha_1,\alpha_2$ are curves (because $X$ is minimal). Since $P$ is 2-tight  we have 
 \[X\setminus\{\alpha_1,\alpha_2\}=Y\setminus\{\alpha_1,\multi{\alpha_2}'\}\subset Z,\] and by applying Lemma \ref{lem:alter:char} to the path  $X\mv Y\mv Z$ we obtain that it is of the form
 \[\{\alpha_1,\alpha_2\}\mv\{\alpha_1,\multi{\alpha_2}'\}\mv\{\multi{\alpha_1}'',\multi{\alpha_2}'\}.\]
\end{proof}

\begin{lemma}\label{tight4gon}
Every quadrangle is 2-tight.
\end{lemma}
\begin{proof}
    Let $\Circ$ be a quadrangle $WXYZ$
     and assume that $X$ is minimal in $\Circ$. We assume $|X|\ge 3$, for otherwise $\Circ$ is obviously 2-tight.
     Let $\fullMv{XY}{\alpha_1}{\mv}{\multi{\alpha_1}'}$. 
     
     If $\alpha_1\in W$, then the path $W\mv X\mv Y$ is  of the form
     \[\{\alpha_1,\multi{\alpha_2}',\alpha_3\}\mv \{\alpha_1,\alpha_2,\alpha_3\}\mv \{\multi{\alpha_1}',\alpha_2,\alpha_3\}.\]
Note that $\alpha_1\notin Z$ and $\alpha_2\notin Z$, for otherwise $Z=X$. If $\Circ$ is not 2-tight, then there is $\alpha_3\in X\setminus\{\alpha_1,\alpha_2\}$ such that $\alpha_3\notin Z$. Then $ZW$ and $YZ$ must be  edges of type 4: 
\[\fullMv{ZW}{\beta}{\mvIVL}{\{\alpha_1,\alpha_3\}},\quad \fullMv{YZ}{\{\alpha_2,\alpha_3\}}{\mvIV}{\beta}\]
 for some $\beta\in Z$. But then $\multi{\alpha_1}'=\alpha_1'$ and $\multi{\alpha_2}'=\alpha_2'$ are curves and
\[Z=\{\alpha'_2,\beta,\dots\}=\{\alpha'_1,\beta,\dots\},\]
which gives $\alpha'_1=\alpha'_2$. This is a contradiction, as $i(\alpha'_1,\alpha_2)=0<i(\alpha'_2,\alpha_2)$.

Now suppose $\alpha_1\notin W$, so the path \[W\mv X\mv Y\] is  of the form:
 \[\{\multi{\alpha_1}'',\alpha_2,\alpha_3\}\mv \{\alpha_1,\alpha_2,\alpha_3\}\mv \{\multi{\alpha_1}',\alpha_2,\alpha_3\}.\]
 If $\Circ$ is not 2-tight, then there are $\alpha_2,\alpha_3\in X\setminus\{\alpha_1\}$ such that $\alpha_2,\alpha_3\notin Z$. 
 Then $ZW$ and $YZ$ must be  edges of type 4: 
\[\fullMv{ZW}{\beta}{\mvIVL}{\{\alpha_2,\alpha_3\}},\quad \fullMv{YZ}{\{\alpha_2,\alpha_3\}}{\mvIV}{\beta}\]
 for some $\beta\in Z$. But then
\[Z=\{\multi{\alpha_1}'',\beta,\dots\}=\{\multi{\alpha_1}',\beta,\dots\},\]
which gives $\multi{\alpha_1}'=\multi{\alpha_1}''$ and $W=Y$, a contradiction.
\end{proof}
\begin{example}\label{ex:5gon}
If $g\geqslant 4$, then there are pentagons in $\pants(N_g)$ that are not 2-tight.
For example, consider the pentagon of the form 
\[\xymatrix@R=0.5pc{
&\{\beta_1,\beta_2,\gamma_1,\gamma_2\}&\{\beta_1,\beta_2,\gamma\}\ar[l]_(.45)4\ar@{-}[dd]^3\\
\{\beta,\gamma_1,\gamma_2\}\ar[ur]^(.45)4\ar[dr]_(.45)4\\
&\{\beta_1,\beta'_2,\gamma_1,\gamma_2\}&\{\beta_1,\beta'_2,\gamma\}\ar[l]^(.45)4
} \]
where $\beta_1,\beta_2,\beta,\gamma_1,\gamma_2,\gamma$ are as in Figure \ref{figNon2Tight4}.
    \begin{figure}[h]
\begin{center}
\includegraphics[width=0.99\customwidth]{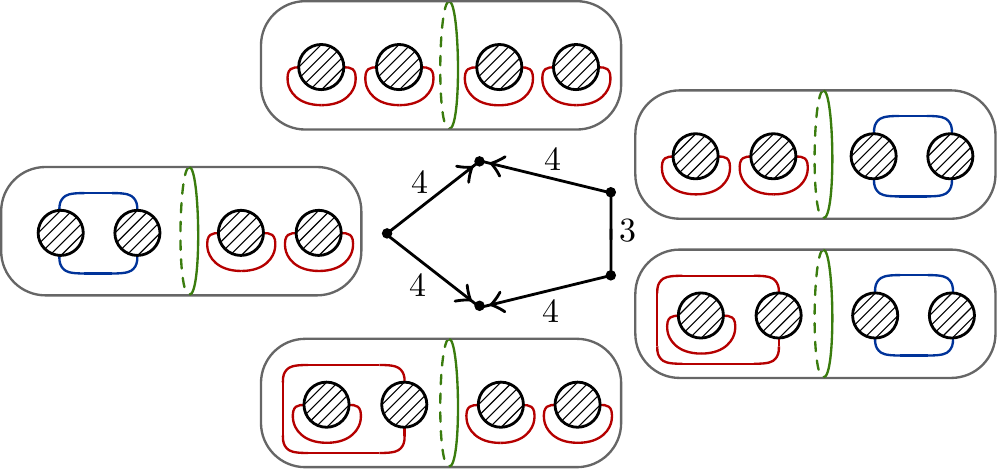} 
\caption{Pentagon in $\pants(N_4)$ that is not 2-tight.}\label{figNon2Tight4} %
\end{center}
\end{figure}
\end{example}

\begin{lemma}\label{two-tight}
Let $\Circ$ be circuit of length at most $6$ without edges of type 4. If $\Circ$ is not 2-tight, then it is an alternating hexagon with even number of edges of type 3.
\end{lemma}
\begin{proof}
    If $\Circ$ does not contain edges of type 4, and we assume that $\Circ$ is not 2-tight, then the proof of Lemma 4 in \cite{Mar} works verbatim and we conclude that $\Circ$ is an alternating hexagon with edges of the form
\begin{align*}
    \alpha_1 \mv \alpha_1',\ \alpha_2 \mv \alpha_2',\ \alpha_3 \mv \alpha_3',\\
    \alpha_1' \mv \alpha_1,\ \alpha_2' \mv \alpha_2,\ \alpha_3' \mv \alpha_3.
\end{align*}
    Hence, the number of edges of type 3 is even -- it is twice the number of  one-sided curves among $\{\alpha_1,\alpha_2,\alpha_3\}$. 
\end{proof}

\begin{defi}
Let $N\notin\{N_{2,1},N_3\}$ and suppose that $\fullMv{XY}{\alpha}{\mvIV}{\{\beta_1,\beta_2\}}$ is a move of type 4.
There is a unique $\gamma\in X\cap Y$ separating a one-holed Klein bottle containing $\alpha,\beta_1$ and $\beta_2$.
We call it \emph{the separating curve associated to $XY$}.
\end{defi}

\begin{rem}\label{rem:asc}
The separating curve associated to $\fullMv{XY}{\alpha}{\mvIV}{\{\beta_1,\beta_2\}}$ is isotopic to the boundary of a regular neighbourhood of $\alpha\cup\beta_1$. In particular, it is determined by $\alpha$ and $\beta_1$.
\end{rem}

\begin{lemma}\label{lemma:bdrpass_new}
Suppose $N\notin\{N_{2,1},N_3\}$ and $P$ is a path in $\pants(N)$ of the form \[X \mvIV Y\mv Z,\] where $YZ$ is of type 1 or 2. Then $Z$ contains the separating curve associated to $XY$ if and only if the path $P$ is contained in a quadrangle. 
\end{lemma}
\begin{proof}
Let 
\[\fullMv{XY}{\alpha_1}{\mvIV}{\{\beta_1,\beta_2\}}\]  and let $\alpha_2$ be the associated separating curve. Since $YZ$ 
is of type 1 or 2, $Z$ contains $\beta_1$ and $\beta_2$. 

If $\alpha_2\in Z$ and 
\[\fullMv{YZ}{\alpha_3}{\mv}{\alpha_3'},\]
then by applying the move 
\[\{\beta_1,\beta_2\}\mvIVL \alpha_1\] to $Z$  we obtain a vertex 
\[W=\{\alpha_1,\alpha_2,\alpha'_3,\ldots \}\] such that $XYZW$ is a a quadrangle.

Conversely, suppose $XYZW$ is a quadrangle and $\alpha_2\notin Z$. We have
\[X=\{\alpha_1,\alpha_2,\dots\},\quad Y=\{\beta_1,\beta_2,\alpha_2,\dots\},\quad Z=\{\beta_1,\beta_2,\alpha'_2,\dots\},\]
for some $\alpha'_2$ such that $i(\alpha'_2,\alpha_2)=2$ ($YZ$ is of type 2 since $\alpha_2$ is separating).
Since $W\ne Y$, we have $\alpha_2\notin W$, and thus $\alpha_1\in W$ (because $\alpha_1\in X$). Thus  $ZW$ is a move o type 4 of the form
\[\fullMv{ZW}{\{\beta_1,\beta_2\}}{\mvIVL}{\alpha_1}.\]
and $W=\{\alpha_1,\alpha'_2,\dots\}$.
By Remark \ref{rem:asc}, the separating curves associated to $ZW$ and $XY$ are isotopic, so $\alpha_2\in W\cap Z$, which is a contradiction, as $i(\alpha'_2,\alpha_2)=2$. 
\end{proof}

\begin{defi}\label{def:7gon}
\emph{Standard heptagon} is a circuit of the form 
\[\xymatrix@R=0.1pc{
    &X\ar[r]^(.46)4&Y\ar@{-}[rd]\\
    &&&Z\ar@{-}[dd]^3\\
    S\ar@{-}[rdd]\ar@{-}[ruu]\\
    &&&P\ar@{-}[dl]\\
    &R\ar[r]_(.46)4&Q
    }\]
such that the path $X\mvIV Y\mv Z$ is not contained in a quadrangle.
\end{defi}
\begin{rem}\label{rem:hept:inv}
By Corollary \ref{cor:type_preserv}, automorphisms of $\pants(N)$ map standard heptagons to standard heptagons.
\end{rem}

The circuit in Figure \ref{fig7gon} is a standard heptagon. It is easy to check, using Lemma \ref{lemma:bdrpass_new}, that it satisfies the definition.
\begin{figure}[h]
\begin{center} 
\includegraphics[width=1\customwidth]{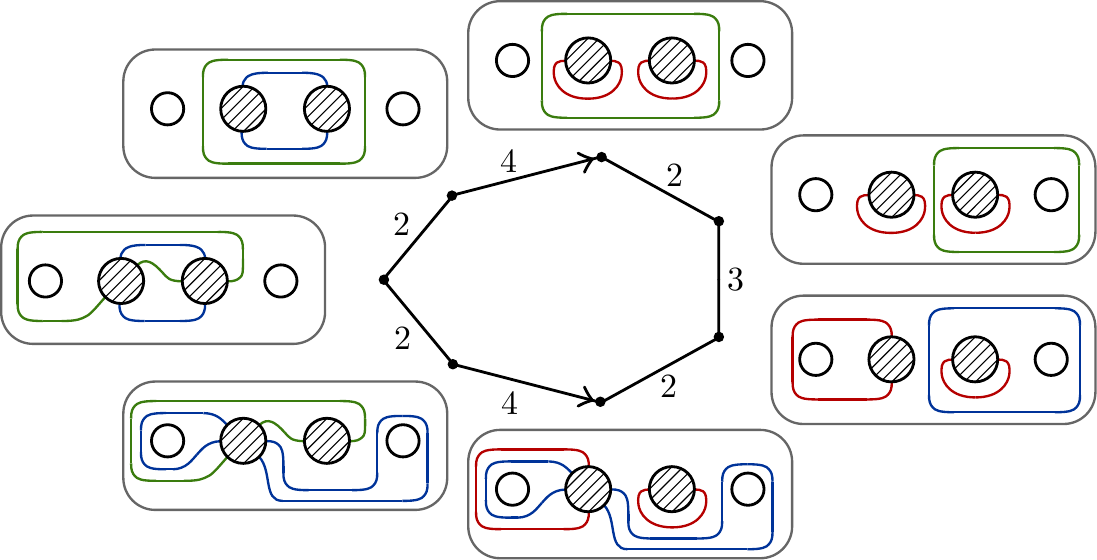}
\caption{Standard heptagon in $\pants(N_{2,2})$.} %
\label{fig7gon}
\end{center}
\end{figure}

\begin{lemma}\label{tight7gon_new}
Every standard heptagon is 2-tight and alternating.    
\end{lemma}
\begin{proof} 
Let $\Circ$ be a standard heptagon as in Definition \ref{def:7gon} and let
\[\fullMv{XY}{\alpha_1}{\mvIV}{\{\beta_1,\beta_2\}}\]
 Since $YZ$ is of type 1 or 2, $Y$ contains at least one two-sided curve, which implies $N\notin\{N_{2,1},N_3\}$. Let $\alpha_2\in X\cap Y$ be the separating curve associated to $XY$. We will show that $X\setminus\{\alpha_1,\alpha_2\}$ is contained in every vertex of $\Circ$.

    By definition of standard heptagon and Lemma \ref{lemma:bdrpass_new}, we have $\alpha_2\notin Z$. Moreover, $\alpha_2$ is separating, so 
    \[\fullMv{YZ}{\alpha_2}{\mvII}{\alpha''_2}\]
     for some $\alpha''_2\in Z$ such that $i(\alpha_2,\alpha''_2)=2$.

    By looking at one-sided curves we see that 
\begin{align*}
Z^-&=Y^-=X^-\cup\{\beta_1,\beta_2\}=S^-\cup\{\beta_1,\beta_2\},\\
P^-&=Q^-=R^-\cup\{\beta'_1,\beta'_2\}=S^-\cup\{\beta'_1,\beta'_2\}.
    \end{align*}
    for some $\beta'_1,\beta'_2\in P^-$. But $ZP$ is an edge of type 3, so we may assume that $\beta'_1=\beta_1$ and
\begin{align*}
&\fullMv{ZP}{\beta_2}{\mvIII}{\beta'_2},\\
&\fullMv{QR}{\{\beta_1,\beta'_2\}}{\mvIVL}{\alpha_1'}
    \end{align*}
    for some $\beta'_2\in P^-$ such that $i(\beta_2,\beta'_2)=1$ and $\alpha'_1\in R^+$.
    
     Let $\alpha'_2\in Q\cap R$ be the separating curve associated to $QR$. By Remark \ref{rem:asc}, $\alpha'_2$ (respectively $\alpha_2$) is the boundary of a regular neighbourhood of $\alpha'_1\cup\beta_1$ (respectively $\alpha_1\cup\beta_1$). It follows that \[\alpha'_2=\alpha_2\iff \alpha'_1=\alpha_1.\]
    
    Suppose that $\alpha'_1=\alpha_1$. Then $\alpha_2=\alpha'_2\in R\cap Q$.  We have $\alpha''_2\in P$ and $i(\alpha''_2,\alpha_2)=2$, so $\alpha_2\notin P$ and $\alpha''_2\notin Q$. It follows that 
    \[\fullMv{PQ}{\alpha''_2}{\mvII}{\alpha_2}\]
     We see that $Y$ and $Q$ differ by the move $\beta_2\mv\beta'_2$, hence $X=R$, a contradiction (with the definition of a circuit).

    So $\alpha'_1\ne \alpha_1$ and $\alpha'_2\ne \alpha_2$. Then $\alpha'_1\notin X$ because all curves in $X\backslash\{\alpha_1\}$ are disjoint from $\beta_1$. Analogously 
    $\alpha_1\notin R$.   
    We have $i(\alpha_2,\alpha'_2)>0$ because $\alpha_2$ and $\alpha'_2$ separate one-holed Klein bottles which are not disjoint, as they both contain $\beta_1$. So $\alpha'_2\notin X$ and $\alpha_2\notin R$.
    By Lemma~\ref{lem:special}, we have $\alpha_1,\alpha'_1\in S$ so  
\begin{align*}
&\fullMv{SX}{\alpha'_1}{\mvII}{\alpha_2}\\
&\fullMv{RS}{\alpha'_2}{\mvII}{\alpha_1}.
\end{align*}
    Summarizing, $\Circ$ has the following form 
\begin{equation}\label{std:hep:proof}
    \xymatrix@R=0.5pc{
    &X\supseteq\{\alpha_1,\alpha_2\}\ar[r]^(.46)4&\{\beta_1,\beta_2,\alpha_2\}\subseteq Y\ar@{-}[rd]^2\\
    &&&\{\beta_1,\beta_2,\alpha''_2\}\subseteq Z\ar@{-}[dd]^3\\
    S\supseteq \{\alpha'_1,\alpha_1\}\ar@{-}[rdd]_2\ar@{-}[ruu]^2\\
    &&&\{\beta_1,\beta'_2,\alpha''_2\}\subseteq P\ar@{-}[dl]^2\\
    &R\supseteq\{\alpha'_1,\alpha'_2\}\ar[r]_(.46)4&\{\beta_1,\beta'_2,\alpha'_2\}\subseteq Q
    }
\end{equation}
    
In particular, $\Circ$ is 2-tight and alternating.
\end{proof}
As an immediate consequence of the explicit form \eqref{std:hep:proof} of the standard heptagon obtained in the above proof, we get the following corollaries.
\begin{cor}
If 
\[R\mv S\mv X\mvIV Y\mv Z\mvIII P\mv Q \mvIVL R\]
is a standard heptagon as in Definition \ref{def:7gon}, then the path 
\[P\mv Q \mvIVL R\]
is not contained in a quadrangle.
\end{cor}
\begin{cor}\label{std:hept:tight:subpath}
    If 
\[\Circ:\, R\mv S\mv X\mvIV Y\mv Z\mvIII P\mv Q \mvIVL R\]
is a standard heptagon, then any subpath
\[P_0\mv P_1\mv P_2\mv P_3\]
of $\Circ$ is 2-tight.
\end{cor}

\begin{example}
If the genus of $N$ is at least $6$, then $\pants(N)$ contains non-standard heptagons of the form
\[
    \xymatrix@R=0.5pc@C=1.2pc{
    &P_0\supseteq\{\alpha_1,\alpha_2,\alpha_3\}\ar[r]^(.46)4&\{\beta_1,\beta_2,\alpha_2,\alpha_3\}\subseteq P_1\ar@{-}[rd]\\
    &&&\{\beta_1,\beta_2,\alpha'_2,\alpha_3\}\subseteq P_2\ar@{-}[dd]^3\\
    P_6\supseteq \{\alpha_1,\alpha_2,\alpha'_3\}\ar@{-}[rdd]\ar@{-}[ruu]\\
    &&&\{\beta'_1,\beta_2,\alpha'_2,\alpha_3\}\subseteq P_3\ar@{-}[dl]\\
    &P_5\supseteq\{\alpha_1,\alpha'_2,\alpha'_3\}\ar[r]_(.46)4&\{\beta'_1,\beta_2,\alpha'_2,\alpha'_3\}\subseteq P_4
    }
\]
 which are not 2-tight. 

Note that if $Q=\{\alpha_1,\alpha'_2,\alpha_3,\dots\}$, then $Q$ is connected by edges to $P_0$, $P_2$, $P_3$ and $P_5$. In particular, the path $P_0\mvIV P_1\mv P_2$ is contained in a quadrangle $P_0P_1P_2Q$, so this heptagon is not standard.
\end{example}

%% file: Edge_in_circuit_new.tex
\section{Tame circuits}\label{sec:tame:circ}
\begin{defi}\label{def:tame:circ} A circuit $\Circ$ of length $n$ is \emph{tame} if $4\le n\le 7$ and
\begin{enumerate}
    \item if $n\in\{5,6\}$, then $\Circ$ contains no edge of type 4;
    \item if $n=6$ and $\Circ$ is alternating, then $\Circ$ contains odd number of edges of type 3;
    \item if $n=7$, then $\Circ$ is a standard heptagon (see Definition \ref{def:7gon}).
\end{enumerate}
\end{defi}

\begin{rem}\label{rem:tame_inv}
    Every tame circuit $\Circ$ is 2-tight by Lemmas \ref{tight4gon}, \ref{two-tight} and \ref{tight7gon_new}. Furthermore, by Corollary \ref{cor:type_preserv} and Remark \ref{rem:hept:inv}, for every 
    $A\in\Aut(\pants(N))$, $A(\Circ)$ is tame, and hence also 2-tight.
\end{rem}

\begin{rem}\label{alter_in _tame}
Suppose that $P$ is an alternating  path
\[W\mv X\mv Y\mv Z,\]  contained in a tame circuit $\Circ$ and assume that $X$ is minimal in $P$.
\begin{enumerate}
    \item If $\Circ$ is of length at most 6, then $X$ is also minimal in $\Circ$, hence $P$ is 2-tight (because $\Circ$ is 2-tight -- see Remark \ref{rem:tame_inv}).
    \item If $\Circ$ is a standard heptagon, then $P$ is 2-tight by Remark \ref{std:hept:tight:subpath}.
\end{enumerate}
In both cases $P$ satisfies the assumptions of Lemma \ref{alter_tight_path}, so it is 
of the form \eqref{eq:2tight_alter}.
\end{rem}

\begin{prop}\label{prop:edge-circuit-new}
    Let $N\ne N_3$ and let 
    \[X=\{\alpha_1,\alpha_2,\alpha_3,\ldots,\alpha_n\}\mathdash \{\alpha_1,\multi{\alpha_2}',\alpha_3,\ldots,\alpha_n\}=X'\]
    be an edge in $\pants(N)$ such that:
    \begin{enumerate}
        \item $|X|\le |X'|$ (i.e. if this is an edge of type 4, then $\alpha_2$ is two-sided and $\multi{\alpha_2}'$ is a pair of disjoint one-sided curves);
        \item if $M$ is the union of non-trivial components of $N_{X\setminus\{\alpha_1,\alpha_2\}}$ then
    $M\ne N_{2,1}$, and if $M\in\{S_{1,2},N_{2,2}\}$ then at least one of the curves $\alpha_1,\alpha_2,\multi{\alpha_2}'$ is separating in $M$.
    \end{enumerate}
  Then the edge $XX'$ is contained in  a tame circuit $\Circ$ 
  \[V \mathdash X\mathdash X' \mathdash W\mathdash\ldots \mathdash V,\]
    such that 
    \begin{enumerate}
        \item the path 
        \(V \mathdash X\mathdash X' \mathdash W\)
         is alternating;
        \item $\alpha_1\not\in V\cup W$;
        \item $X$ is minimal in the path \(V \mathdash X\mathdash X' \mathdash W\).
        \end{enumerate}
\end{prop}
\begin{proof}
Let $M$ be the union of non-trivial components of $N_{X\setminus\{\alpha_1,\alpha_2\}}$. 

\medskip
{\bf Case 1: $M$ is disconnected.}
If $M_1$ and $M_2$ are connected components of $M$, then 
\[M_1,M_2\in\left\{S_{0,4},S_{1,1},N_{1,2},N_{2,1}\right\}\]
and $XX'$ is contained in an alternating quadrangle $\Circ$ of the form
\[\xymatrix{ 
    \{\alpha_1,\alpha_2\}\ar@{-}[r]\ar@{-}[d]&\{\alpha_1,\multi{\alpha_2}'\}\ar@{-}[d]\\
    \{\multi{\alpha_1}',\alpha_2\}\ar@{-}[r]&\{\multi{\alpha_1}',\multi{\alpha_2}'\}
    }\]
Note that $\multi{\alpha_1}'$ and $\multi{\alpha_2}'$ may be pairs of disjoint one-sided curves (corresponding to edges of types 4), but in each case $X$ is minimal in $\Circ$.

\medskip
{\bf Case 2: $M$ is connected and $M\setminus \{\alpha_1,\alpha_2\}$ has three connected components.} In this case $M\cong S_{0,5}$, and we can use standard 5-cell to construct $\Circ$ as in the proof of Lemma~6 of \cite{Mar}. Since $M$ is orientable, $\Circ$ has no edges of type 4, so $X$ is minimal in $\Circ$.

\medskip
{\bf Case 3: $M$ is connected and $M\setminus \{\alpha_1,\alpha_2\}$ has less than 3 connected components} If $M\setminus \{\alpha_1,\alpha_2\}$ has less than 3 connected components, then it has exactly 2 connected components. Indeed, otherwise $M=N_{2,1}$ or $M=N=N_3$, which contradicts our assumptions.

If $M\cong S_{1,2}$, then  the  situation is identical to that in the proof of Lemma~6 of \cite{Mar}. Following the same argument  we obtain a hexagon $\Circ$ without edges of type 3 or 4 (because $M$ is orientable). This hexagon is not alternating, and hence $\Circ$ is tame and $X$ is minimal in $\Circ$.

Suppose now that $M\cong N_{1,3}$. If $\alpha_1$ is two-sided, then it is separating in $M$ and $\alpha_2,\multi{\alpha_2}'=\alpha_2'$ are one-sided curves in $M\setminus\{\alpha_1\}\cong N_{1,2}$ connected by an edge of type 3. If $\alpha_1$ is one-sided, then $\alpha_2,\multi{\alpha_2}'=\alpha_2'$ are two-sided curves in $M\setminus\{\alpha_1\}\cong S_{0,4}$ connected by an edge of type 2. In either case $XX'$ is contained in an alternating hexagon $\Circ$ without edges of type 4 and with 3 edges of type 3 -- see Figure \ref{fig:6gon}. 
\begin{figure}
    \centering
    \includegraphics[width=0.8\customwidth]{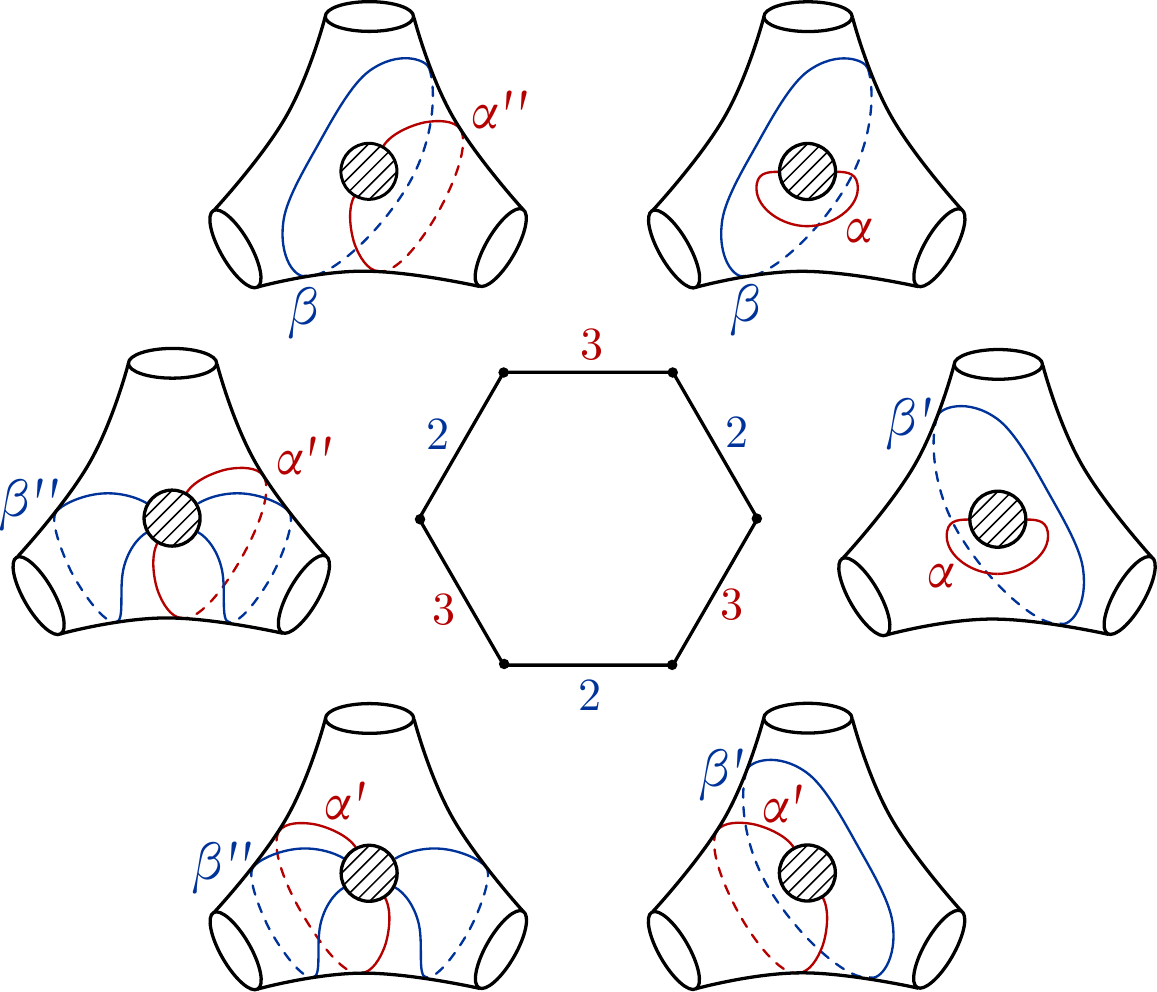}
    \caption{Tame alternating hexagon in $\pants(N_{1,3})$.}
    \label{fig:6gon}
\end{figure}
Such $\Circ$ is tame, and $X$ is minimal in $\Circ$.

Finally, suppose that $M\cong N_{2,2}$. We assumed that one of $\alpha_1,\alpha_2$ or $\multi{\alpha_2}'$ is separating. Moreover, $M\setminus\{\alpha_1,\alpha_2\}$ are 2 pairs of pants, so $\alpha_1$, $\alpha_2$ or $\alpha_2'$ cuts $N_{2,1}$ from $M$. 

If $\alpha_1$ cuts $N_{2,1}$ from $M$, then $\alpha_2,\multi{\alpha_2}'\in N_{2,1}$ are connected by an edge of type 3 or 4. In either case $XX'$ is contained in a standard heptagon $\Circ$. If $XX'$ is of type 4, then $X$ is minimal in $\Circ$. 

If $\alpha_2$ (or $\alpha_2'$) cuts $N_{2,1}$ from $M$, then $\alpha_1$ is the unique two-sided curve in $N_{2,1}$ and $\alpha_2,\multi{\alpha_2}'=\alpha_2'$ are connected by an edge of type 2. As in the previous case, $XX'$ is contained in a standard heptagon $\Circ$ and
$X$ is minimal in $\Circ$  (in this case the edges $VX$ and $X'W$ are of different type; one is of type 4 and the other one is of type 2).
\end{proof}
\begin{prop}\label{prop:edge-circuit-exception}
    Let
    \[X=\{\alpha_1,\alpha_2,\alpha_3,\ldots,\alpha_n\}\mathdash\{\alpha_1,\multi{\alpha_2}',\alpha_3,\ldots,\alpha_n\}=X'\]
    be an edge in $\pants(N)$ and let $M$ be the union of non-trivial components of $N_{X\setminus\{\alpha_1,\alpha_2\}}$. If $M\in\{S_{1,2},N_{2,2}\}$ and each of $\{\alpha_1,\alpha_2,
        \multi{\alpha_2}'\}$ is nonseparating in $M$,
      then $\multi{\alpha_2}'=\alpha_2'$ is a curve and there exists a vertex 
    \[X''=\{\alpha_1,\alpha_2'',\alpha_3,\ldots,\alpha_n\}\]
    in $\pants(N)$ and two edges 
    \begin{align*}
        XX''&\,:\, \{\alpha_1,\alpha_2,\alpha_3,\ldots,\alpha_n\}\mathdash\{\alpha_1,\alpha_2'',\alpha_3,\ldots,\alpha_n\}\\
        X''X'&\,:\, \{\alpha_1,\alpha_2'',\alpha_3,\ldots,\alpha_n\}\mathdash\{\alpha_1,\alpha_2',\alpha_3,\ldots,\alpha_n\}
    \end{align*}
    such that $XX''$ and $X''X'$ satisfy the assumptions of Proposition \ref{prop:edge-circuit-new}.
\end{prop}
\begin{proof}
    If $M\cong S_{1,2}$ see Figure 8 of \cite{Mar} for the construction of $\alpha_2''$, so assume that $M\cong N_{2,2}$. Note the following:
    \begin{enumerate}
        \item $\{\alpha_1,\alpha_2\}$ and $\{\alpha_1,\multi{\alpha_2}'\}$ decompose $M$ into pairs of pants, hence $\alpha_1$, $\alpha_2$ and $\multi{\alpha_2}'=\alpha_2'$ are two-sided curves;
        \item $\alpha_2$ and $\alpha_2'$ intersect in two points in $M\setminus\{\alpha_1\}\cong S_{0,4}$.
    \end{enumerate}
Hence, the configuration of $\alpha_1,\alpha_2,\alpha_2'$ is as in Figure \ref{fig:n22:nonsep}.
    \begin{figure}[h]
\begin{center}
\includegraphics[width=0.55\customwidth]{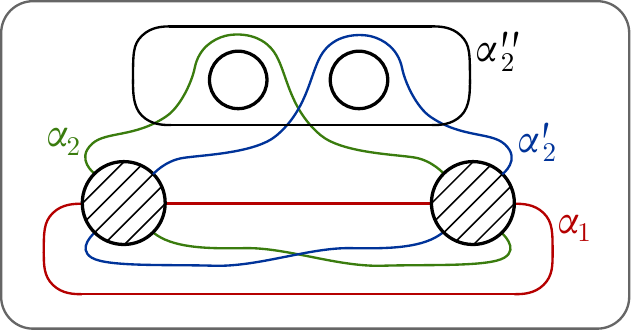}
\caption{Nonseparating curves $\alpha_1,\alpha_2,\alpha_2'$ in $N_{2,2}$.}\label{fig:n22:nonsep} %
\end{center}
\end{figure}
See the same figure for $\alpha_2''$.
\end{proof}

%% file: aut_pants_new.tex
\section{Automorphisms of $\pants(N)$}\label{sec:aut:pants}
\begin{defi}
    If $(F,X)$ is a marked $F$-graph and $\alpha\in X$ is the unique curve such that $F= F(X,\alpha)$ (see Proposition \ref{lem:Fsubgraph}), then define 
    \[v[F,X]=\alpha.\]
\end{defi}
 \begin{rem}\label{lem:v:marked}
 By definition, if $X\in \pants^0(N)$ and $\alpha\in X$, then
     \[v\left[F(X,\alpha),X\right]=\alpha.\]
 \end{rem}
\begin{defi}
    Let $A\in \Aut(\pants(N))$ and $\alpha\in \curv^0(N)$. Extend $\alpha$ to a vertex $X$ of $\pants(N)$ and define 
    \[\Phi(A)(\alpha)=v\left[A(F(X,\alpha)),A(X)\right].\]
\end{defi}
\begin{prop}\label{prop:Phi_well_defined}
    Definition of $\Phi$ does not depend on $X\in \pants^0(N)$, hence $\Phi(A)$ is a well defined map
    \[\Phi(A):\,  \curv^0(N)\to \curv^0(N).\]
\end{prop}
\begin{proof}
    We need to show that if $X,X'\in \pants^0(N)$ are such that $\alpha\in X\cap X'$, then 
    \[v\left[A(F(X,\alpha)),A(X)\right]=v\left[A(F(X',\alpha)),A(X')\right].\]
    Since $\pants(N_\alpha)$ is path-connected by Theorem \ref{pants_connected}, it suffices to show the statement for $X$ and $X'$ being connected by an edge $X\mv X'$.
    We can assume that
    \begin{align*}
            X&=\{\alpha,\alpha_2,\alpha_3,\ldots,\alpha_n\},\\
            X'&=\{\alpha,\multi{\alpha_2}',\alpha_3,\ldots,\alpha_n\},
        \end{align*}
         so if $X$ and $X'$ are connected by an edge of type 4, we assume that $\alpha_2$ is two-sided. 
         Let $M$ be the union of non-trivial components of $N_{X\setminus\{\alpha,\alpha_2\}}$. We consider three cases.
  
\medskip
{\bf Case 1: $N\ne N_3$ and $M\ne N_{2,1}$.}
        By Proposition \ref{prop:edge-circuit-exception}, we can assume that $X$ and $X'$ satisfy the assumptions of Proposition~\ref{prop:edge-circuit-new} -- if this is not the case, we split the edge 
        \[X\mv X'\]
        into two edges 
        \[X\mv X''\mv X'\]
        satisfying the assumptions of Proposition \ref{prop:edge-circuit-new}.

        By Proposition \ref{prop:edge-circuit-new}, $X$ and $X'$ are consecutive vertices of a tame circuit $\Circ$ in $\pants(N)$, containing an alternating path $P$
        \[V \mv X\mv X' \mv W\]
        such that $\alpha\notin V\cup W$, and $X$ is minimal in $P$.
        Since, $\alpha\notin V\cup W$ and $\alpha\in X\cap X'$, by Corollary~\ref{cor:edge:in:F},
    \begin{equation}\label{eq:FXV}
    \begin{aligned}
        F(XV)&=F(X,\alpha)\\
        F(X'W)&=F(X',\alpha).
        \end{aligned}
    \end{equation}        
        By Remark \ref{rem:tame_inv}, $A(\Circ)$ is tame and 2-tight. 
        By Remark \ref{rem:alter:inv}, the path $A(P)$ 
        \[A(V) \mv A(X)\mv A(X') \mv A(W)\]
        is alternating, and by Lemma \ref{lem:minimal}, $A(X)$ is minimal in $A(P)$.
        By Remark \ref{alter_in _tame}, $A(P)$ satisfies the assumptions of Lemma \ref{alter_tight_path} and we have
\[\xymatrix{ 
    A(V)\supseteq \{\multi{\beta_1}',\beta_2\}\ar@{-}[r]&\{\beta_1,\beta_2\}\subseteq A(X)\ar@{-}[d]\\
    A(W)\supseteq \{\multi{\beta_1}'',\multi{\beta_2}'\}\ar@{-}[r]&\{\beta_1,\multi{\beta_2}'\}\subseteq A(X')
    }\]
        where $\beta_1$ and $\beta_2$ are curves. By Corollary \ref{cor:edge:in:F}, 
    \begin{equation}\label{eq:FAXAV}
    \begin{aligned}
        F(A(X)A(V))&=F(A(X),\beta_1)\\
        F(A(X')A(W))&=F(A(X'),\beta_1).
    \end{aligned}\end{equation}
    Hence, Remark \ref{lem:v:marked}, Corollary \ref{A:inv:F:graph}, \eqref{eq:FXV} and \eqref{eq:FAXAV} imply that
         \begin{align*}
        v\left[A(F(X,\alpha)),A(X)\right]&=v\left[A(F(XV)),A(X)\right]\\
        &=v\left[F(A(X)A(V)),A(X)\right]=v\left[F(A(X),\beta_1),A(X)\right]=\beta_1\\
v\left[A(F(X',\alpha)),A(X')\right]&=v\left[A(F(X'W,X')),A(X')\right]\\
        &=v\left[F(A(X')A(W)),A(X')\right]=v\left[F(A(X'),\beta_1),A(X')\right]=\beta_1.
         \end{align*}
\medskip
{\bf Case 2: $M=N_{2,1}$.} In this case $\alpha$ and $\alpha_2$ are one-sided curves contained in a one-holed Klein bottle $M$.
Let $V,W\in\pants^0(N)$ be such that
    \[F(X,\alpha)=XV,\qquad F(X',\alpha)=X'W.\]
    We have a path $P$ in  $\pants(N)$:
    \[V \mvIII X\mvIII X' \mvIII W,\]  
    which is contained in a fan subgraph. Then $A(P)$ is also a path of edges of type 3 contained in a fan subgraph of $\pants(N)$, and by the structure of $\pants(N_{2,1})$ (see Proposition \ref{pantsN21}) it is of the same form as in Case 1 with one-sided $\beta_1,\beta_2,\beta_1'$, $\beta_1''$ and $\beta_2'$. The rest of the proof is the same as in Case 1.
    
\medskip
{\bf Case 3: $N=N_3$.} We can assume that one of the curves $\alpha$ or $\alpha_2$ is two-sided, because otherwise we are in Case 2. We have $X=\{\alpha,\alpha_2\}$. Extend $XX'$ to an alternating path
$P$
        \[V \mathdash X\mathdash X' \mathdash W\]
        such that $\alpha\notin V\cup W$.  By Remark \ref{rem:alter:inv}, the path $A(P)$ is alternating. Since one of the edges $A(XX')$ or $A(XV)$ is not of type $3$ (because $\alpha$ or $\alpha_2$ is two-sided), $A(X)$ contains a two-sided curve and $|A(X)|=2$, which means that $A(P)$ is 2-tight and $A(X)$ is minimal.  By Lemma \ref{alter_tight_path}, $A(P)$ is of the same form as in Case 1, and the rest of the proof is the same as in Case~1.
\end{proof}
Our next goal is to prove the following proposition.
\begin{theorem}\label{prop:Phi:iso}
    For every nonorienatable surface $N$ with $\chi(N)<0$, 
    \[\Phi\colon \Aut(\pants(N))\to \Aut(\curv(N))\]
    is an isomorphism.
\end{theorem}
The proof of the above proposition is divided into several steps -- see Lemmas \ref{lem:Phi:simp} -- \ref{lem:Phi:sur} below.
\begin{lemma}\label{lem:Phi:simp}
    If $A\in \Aut(\pants(N))$ and $\alpha,\beta\in \curv^0(N)$ are such that $i(\alpha,\beta)=0$, then \[i(\Phi(A)(\alpha),\Phi(A)(\beta))=0.\] Hence, 
    \[\Phi(A)\colon\curv(N)\to \curv(N)\]
     is a simplicial map. 
\end{lemma}
\begin{proof}
    We extend $\{\alpha,\beta\}$ to a vertex
    \[X=\{\alpha,\beta,\gamma_3,\ldots,\gamma_n\}\in \pants^0(N).\]
    Then 
    \[X=F(X,\alpha)\cap F(X,\beta)\]
    and 
\begin{align*}
    \Phi(A)(\alpha)&=v\left[A(F(X,\alpha)),A(X)\right]\\
    \Phi(A)(\beta)&=v\left[A(F(X,\beta)),A(X)\right].
\end{align*}
Hence,
\[ \Phi(A)(\alpha),\Phi(A)(\beta)\in A(X)\]
are disjoint as they are in the same vertex $A(X)$ of $\pants^0(N)$.
\end{proof}


\begin{lemma}\label{lem:Phi:hom}
    If $A,B\in \Aut(\pants(N))$ and $\alpha\in \curv^0(N)$, then
    \[  \Phi(A)(\Phi(B)(\alpha))=\Phi(AB)(\alpha).\]
    Hence 
    \[\Phi\colon \Aut(\pants(N))\to \Aut(\curv(N))\]
    is a homomorphism.
\end{lemma}
\begin{proof}
    Let $X\in \pants^0(N)$ such that $\alpha\in X$ and $F=F(X,\alpha)$. If
    \[\beta=\Phi(B)(\alpha)=v(B(F),B(X)),\]
    then $\beta\in B(X)$ and 
    \[B(F)=F(B(X),\beta).\]
    Hence,
    \[\begin{aligned}
      \Phi(A)(\Phi(B)(\alpha))&=\Phi(A)(\beta)=v\left[A(F(B(X),\beta)),A(B(X))\right]\\&=v\left[AB(F),AB(X)\right]=
    v\left[AB(F(X,\alpha)),AB(X)\right]=\Phi(AB)(\alpha).  
    \end{aligned}\]
    This and Lemma \ref{lem:Phi:simp} imply that for any $A\in \Aut(\pants(N))$, $\Phi(A)$ is an invertible simplicial map, so $\Phi(A)\in \Aut(\curv(N))$.
\end{proof}
\begin{lemma}[Injectivity] \label{lem:Phi:inj}
    If $\Phi(A)=\text{id}$ in $\Aut(\curv(N))$, then $A=\text{id}$ in $\Aut(\pants(N))$.
\end{lemma}
\begin{proof}
For every $X\in\pants^0(N)$ and $\alpha\in X$ we have
\[\alpha=\Phi(A)(\alpha)=v\left[A(F(X,\alpha)),A(X)\right]\in A(X),\]
so $X\subseteq A(X)$. But $X$ is a maximal multicurve, so $X=A(X)$.
\end{proof}
\begin{lemma}[Surjectivity] \label{lem:Phi:sur}
    For every nonorienatable surface $N$ with $\chi(N)<0$, 
    \[\Phi\colon \Aut(\pants(N))\to \Aut(\curv(N))\]
    is surjective.
\end{lemma}
\begin{proof}
    Consider the diagram
    \begin{equation}\label{diagram}
     \begin{CD}
        \mcg(N) @= \mcg(N)\\
        @V{\theta}VV @VV{\eta}V\\
        \Aut(\pants(N)) @>\Phi>> \Aut(\curv(N))
    \end{CD}   
    \end{equation}
    By Theorem \ref{main:aut_curv}, $\eta$ is surjective, so it is enough to prove that this diagram commutes.

    Let $f\in \mcg(N)$, $\alpha\in \curv^0(N)$ and extend $\alpha$ to 
    \[X=\{\alpha,\alpha_2,\ldots,\alpha_n\}\in \pants^0(N).\]
    Then 
    \[\theta(f)(X)=
    \{f(\alpha),f(\alpha_2),\ldots,f(\alpha_n\}=f(X)\] and
    \[\theta(f)(F(X,\alpha))=F(f(X),f(\alpha)).\]
    By Remark \ref{lem:v:marked}, 
    \[\Phi(\theta(f))(\alpha)=v\left[F(f(X),f(\alpha)),f(X)\right]=f(\alpha)=\eta(f)(\alpha).\]
\end{proof}

\begin{proof}[Proof of Theorem \ref{main:aut_pants}]
Theorem \ref{main:aut_pants} follows from commutativity of the diagram \eqref{diagram}, Theorem \ref{main:aut_curv} and the fact that
$\Phi\colon \Aut(\pants(N))\to \Aut(\curv(N))$ is an isomorphism -- see Theorem~\ref{prop:Phi:iso}.
\end{proof}

%% file: Iso_pants.tex
\section{Isomorphisms of pants graphs}\label{sec:iso:pants}
In this section we prove that if two nonorientable surfaces have isomorphic pants graphs then they are homeomorphic (Theorem \ref{prop:iso:pants}). 

\medskip
For $A\in\pants^0(N)$ the number $|A^-|$ is called {\it one-sided degree} of $A$ and denoted by $d^-(A)$.
\begin{rem}
The maximum one-sided degree of a vertex of $\pants(N)$ is equal to the genus of $N$.
\end{rem}
\begin{lemma}\label{lem:dminus}
    Let $A\in \pants^0(N)$. There are exactly $d^-(A)$ vertices in $\pants(N)$ connected to $A$ by an edge of type 3.
\end{lemma}
\begin{proof}
For every $\alpha\in A^-$ there is a unique one-sided curve $\alpha'$, disjoint from $A\setminus\{\alpha\}$ and such that $i(\alpha,\alpha')=1$. Indeed, the nontrivial component of the complement of $A\setminus\{\alpha\}$ in $N$ is a two-holed projective plane, and there are only two isotopy classes of nontrivial curves in $N_{1,2}$.    
\end{proof}

\begin{lemma}\label{lem:Farey_number}
If $A$ is vertex of $\pants(N_{g,b})$ such that $|A^-|=g$ then $A$ belongs to exactly $g+b-3$ Farey subgraphs.
\end{lemma}
\begin{proof}
 The complement of $A^-$ in $N_{g,b}$ is a $(g+n)$-holed sphere, so $A^+$ consists of $g+b-3$ non-special two-sided curves.
 If $A$ belongs to a Farey subgraph $F$ then $F=F(A,\alpha)$ for a unique $\alpha\in A^+$ (Proposition \ref{lem:Fsubgraph}).
 It follows that $A$ belongs to exactly $g+n-3$ Farey subgraphs.  
\end{proof}

 \begin{theorem}\label{prop:iso:pants}
    If $N$ and $N'$ are non-orientable surfaces such that $\pants(N)$ is isomorphic to $\pants(N')$ then $N$ and $N'$ are homeomorphic. Furthermore, every isomorphism of pants graphs is induced by a homeomorphism of the underlying surfaces.
 \end{theorem}
 \begin{proof}
     Let $N=N_{g,b}$, $N'=N_{g',b'}$ and let $\varphi\colon\pants(N)\to\pants(N')$ be an isomorphism. Since $\varphi$ preserves type of edges (Corollary \ref{cor:type_preserv}), it also preserves one-sided degree of vertices by Lemma \ref{lem:dminus}. Let $A$ be a vertex of $\pants(N)$ of maximum one-sided degree. Then $\varphi(A)$ is also a vertex of maximum one-sided degree in  $\pants(N')$ and
     \[g=d^-(A)=d^-(\varphi(A))=g'.\] 
     By Lemma \ref{lem:Farey_number}, $A$ belongs to $g+b-3$ Farey subgraphs, whereas $\varphi(A)$ belongs to $g'+b'-3$ Farey subgraphs. Since $\varphi$ maps Farey subgraphs to Farey subgraphs, we have
     \[g+b-3=g'+b'-3,\]
     which implies $b=b'$.

     The second part of the statement follows from Theorem \ref{main:aut_pants}.
 \end{proof}

%% file: aut_curve.tex
\section{Automorphisms of $\curv(N)$} \label{sec:aut:curves}
The purpose of this section is to complete the description of $\Aut(\curv(N))$ given in \cite{AK} by proving the following result.
\begin{theorem}\label{thm:aut_curv}
    Let $N=N_{g,4-g}$, where $1\le g\le 4$. The natural map \[\mcg(N)\to\Aut(\curv(N))\] is an isomorphism.
\end{theorem}

The \emph{star} of a vertex $\alpha\in\curv^0(N)$, denoted by $\St(\alpha)$ is the subcomplex of $\curv(N)$ consisting of all simplices  containing $\alpha$ and all faces of such simplices. The \emph{link} of $\alpha$, denoted by 
    $\Lk(\alpha)$ is the subcomplex of $\St(\alpha)$ consisting of all simplices of $\St(\alpha)$
not containing $\alpha$. We have $\Lk(\alpha)\cong\curv(N_\alpha)$.
Following \cite{AK}, we say that a one-sided curve $\alpha$ is \emph{essential} if either $g=1$ or $N_\alpha$ is nonorientable. We denote by $\curv^0_e(N)$ the set of essential one-sided vertices of $\curv(N)$, and for $g\ge 2$ we denote by $\curv^1_e(N)$ the subgraph of $\curv^1(N)$ spanned by $\curv^0_e(N)$. 

\begin{lemma}\label{AK:X}
    For $g\ge 2$ and $b\ge 0$ the graph  $\curv^1_e(N)$ is connected.
\end{lemma}
\begin{proof}
    We use the graph $X(N)$ defined in \cite{AK}, which has $\curv^0_e(N)$ as the vertex set and two vertices $\alpha,\beta$ are joined by en edge in $X(N)$ if and only if $i(\alpha,\beta)=1$. If $\alpha,\beta\in\curv^0_e(N)$ are such that $i(\alpha,\beta)=1$ then the complement of a regular neighbourhood of $\alpha\cup\beta$ in $N$ is nonorientable, and so there is an essential one-sided curve $\gamma$ disjoint from $\alpha$ and $\beta$. It follows that if two vertices are connected by a path in  $X(N)$ then they are also connected by a path in  $\curv^1_e(N)$. From this observation and the description of connected components of $X(N)$ given in \cite{AK} it follows immediately that 
    $\curv^1_e(N)$ is connected.
\end{proof}

\begin{prop}[Theorem 5.1 of \cite{AK}]\label{AK:top_type}
If $\alpha\in\curv^0(N)$ and $A\in\Aut(\curv(N))$ then $\alpha$ and $A(\alpha)$ are in the same $\mcg(N)$-orbit.
 \end{prop}

We are ready to prove the main result of this section.
 \begin{proof}[Proof of Theorem \ref{thm:aut_curv}]
 We use induction on the genus of $N$. For $g=1$ the theorem is true by \cite{Szep}.
 Let $g\ge 2$ and assume the theorem is true for surfaces of lower genus. Let $A\in\Aut(\curv(N))$. We need to show that there is a unique
 $f\in\mcg(N)$ such that $f(\alpha)=A(\alpha)$ for every $\alpha\in\curv^0(N)$. We break the proof into three steps.

 \medskip
 \noindent{\bf Step 1.} {\it For every $\alpha\in\curv^0_e(N)$ there is a unique
 $f_\alpha\in\mcg(N)$ such that $f_\alpha(\beta)=A(\beta)$ for every $\beta\in\St(\alpha)$.}

 \medskip\noindent
Let $\alpha\in\curv^0_e(N)$. By Proposition \ref{AK:top_type}, we have $A(\alpha)\in\curv^0_e(N)$
\[N_{A(\alpha)}\cong N_\alpha\cong N_{g-1,5-g}.\]
By induction hypothesis there is a unique
$f'_\alpha\colon N_\alpha\to N_{A(\alpha)}$ such that $f'_\alpha(\beta)=A(\beta)$ for every $\beta\in\Lk(\alpha)=\curv(N_\alpha)$.
We need to show that $f'_\alpha$ can be extended over $N$ to a unique $f_\alpha\in\mcg(N)$ such that $f_\alpha(\alpha)=A(\alpha)$. It suffices to show that $f'_\alpha$ maps the boundary component of $N_\alpha$ corresponding to $\alpha$ to the boundary component of $N_{A(\alpha)}$ corresponding to $A(\alpha)$, which is equivalent to $f'_\alpha(\partial N)=\partial N$. There is nothing to prove if $g=4$, since then $\partial N=\emptyset$.

Assume $g=2$ and let $\beta$ be a separating curve disjoint from $\alpha$ such that one of the components of $N_\beta$, say $M$, is a pair of pants containing $\partial N$ (Figure \ref{F1}). 
 \begin{figure}[h]
\begin{center}
\includegraphics[width=0.83\customwidth]{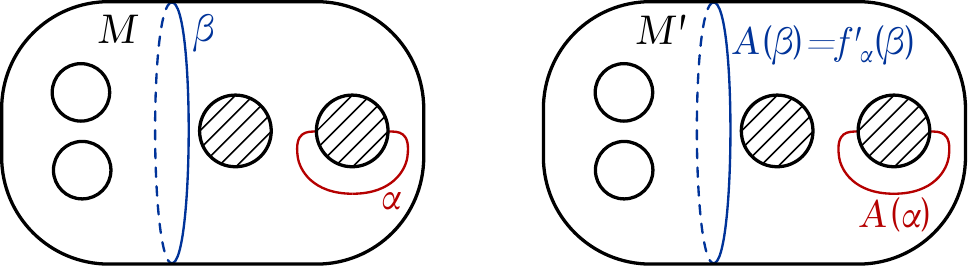} 
\caption{Curves in $N_{2,2}$ from the proof of Step 1.}\label{F1} %
\end{center}
\end{figure}
By Proposition~\ref{AK:top_type}, $A(\beta)$ is also a separating curve disjoint from $A(\alpha)$ such that one of the components of $N_{A(\beta)}$, say $M'$, is a pair of pants containing $\partial N$. Observe that $M$ (respectively $M'$) is the unique orientable component of $N_{\{\alpha,\beta\}}$ (respectively $N_{\{A(\alpha),A(\beta)\}}$).
Since $A(\beta)=f'_\alpha(\beta)$ we have $f'_\alpha(M)=M'$ and $f'_\alpha(\partial N)=\partial N$.

Finally assume $g=3$ and let $\beta$ be a separating curve disjoint from $\alpha$ such that both components of $N_\beta$ are nonorientable and $\alpha$ is contained in the genus $2$ component of $N_\beta$. Let $M$ be the genus $1$ component of $N_\beta$, $\gamma$ a one-sided curve in $M$, and $\delta$ the unique two-sided nonseparating curve in $N\setminus M$ (Figure \ref{F2}). 
\begin{figure}[h]
\begin{center}
\includegraphics[width=0.83\customwidth]{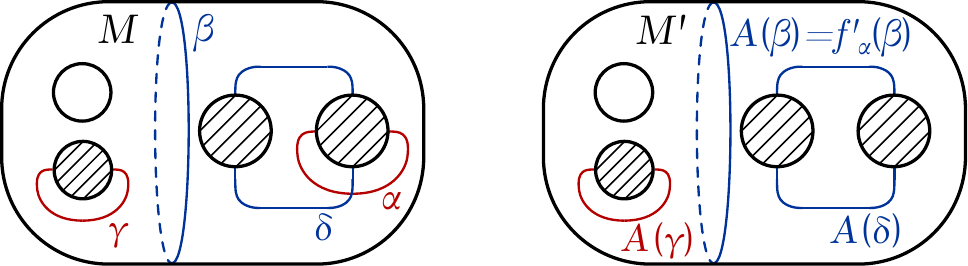} 
\caption{Curves in $N_{3,1}$ from the proof of Step 1.}\label{F2} %
\end{center}
\end{figure}
By Proposition \ref{AK:top_type}, $A(\beta)$ is also a separating curve disjoint from $A(\alpha)$ such that both components of $N_{A(\beta)}$ are nonorientable. Let $M'$ be the genus $1$ component of $N_{A(\beta)}$ which must contain $\partial N$. By Proposition \ref{AK:top_type}, $A(\delta)$ is a two-sided nonseparating curve disjoint form $A(\beta)$, so it must be contained in $N\setminus M'$. Furthermore, $A(\delta)$ fills $N\setminus M'$ which implies that $A(\gamma)$ is contained in $M'$. Since $A(\beta)=f'_\alpha(\beta)$ and $A(\gamma)=f'_\alpha(\gamma)$ we have $f'_\alpha(M)=M'$ and $f'_\alpha(\partial N)=\partial N$.

 \medskip
 \noindent{\bf Step 2.} {\it For every $\alpha,\beta\in\curv^0_e(N)$ we have $f_\alpha=f_\beta$.}

 \medskip\noindent
Since $\curv^1_e(N)$ is connected by Lemma \ref{AK:X} we can assume that $\alpha$ and $\beta$ are disjoint.
Then 
\[f_\alpha(\alpha)=A(\alpha)=f_\beta(\alpha),\quad f_\alpha(\beta)=A(\beta)=f_\beta(\beta).\]
Let $h=f_\alpha\circ f^{-1}_\beta$, $M=N_{\{\alpha,\beta\}}$ and $h'=h|_M$. Observe that $h'$ fixes every vertex of $\curv(M)=\Lk(\{\alpha,\beta\})$ because $f_\alpha(\gamma)=f_\beta(\gamma)=A(\gamma)$ for every $\gamma\in\Lk(\{\alpha,\beta\})$. If $M$ is nonorientable then $M\cong N_{g-2,6-g}$ and $h'=id_M$ by the induction hypothesis. If $M$ is orientable then either $g=2$ and $M\cong S_{0,4}$, or $g=4$ and $M\cong S_{1,2}$.  Let $H$ be the kernel of the map $\mcg(M)\to\Aut(\curv(M))$, which is generated by hyperelliptic involutions, $H\cong\Z_2\oplus\Z_2$ if $M\cong S_{0,4}$, or $H\cong\Z_2$ if $M\cong S_{1,2}$, see \cite{Luo}.
We know that $H$ acts faithfully on the components of $\partial M$, and in the case of $S_{0,4}$ it acts by even permutations. Since $h'$ fixes the two components of $\partial M$ corresponding to $\alpha$ and $\beta$, we have $h'=id_M$. It follows that $h=id$.

\medskip
 \noindent
 After Step 2 we know that there is a unique $f\in\mcg(N)$ such that for every $\beta\in\curv^0_e(N)$ and $\alpha\in\St(\beta)$ we have $f(\alpha)=A(\alpha)$.
 
\medskip
 \noindent{\bf Step 3.} {\it For every $\alpha\in\curv^0(N)$ we have $f(\alpha)=A(\alpha)$.}

 \medskip
 \noindent
 We know that $f(\alpha)=A(\alpha)$ if either
 \begin{itemize}
     \item[(1)] $N_\alpha$ is connected and nonorientable, or
     \item[(2)] $N_\alpha$ has two nonorientable components, or
     \item[(3)] $N_\alpha$ has a nonorientable component of genus at least $2$.
 \end{itemize}
 Indeed, in each of these cases there is $\beta\in\curv^0_e(N)$ such that $\alpha\in\St(\beta)$. It remains to consider the following cases.
\begin{itemize}
     \item[(4)] $N_\alpha$ is connected and orientable.
     \item[(5)] $N_\alpha\cong S_{1,1}\sqcup N_{1,2}$.
\end{itemize}
Suppose $N_\alpha$ is connected and orientable. Let $\beta$ be any curve in $N$ such that $i(\alpha,\beta)=1$. Observe that $\beta$ is one-sided and essential, so for every $\gamma\in\St(\beta)$ we have $A(\gamma)=f(\gamma)$. By \cite[Lemma 7.5]{AK} we also have 
$A(\gamma)=f(\gamma)$ for every $\gamma\in\curv^0(N)$ such that $i(\beta,\gamma)=1$, so in particular
$A(\alpha)=f(\alpha)$.



Finally suppose $N_\alpha\cong S_{1,1}\sqcup N_{1,2}$. Consider the curves $\beta$, $\gamma$ and $\delta$ in Figure \ref{F5}. 
\begin{figure}[h]
\begin{center}
\includegraphics[width=0.38\customwidth]{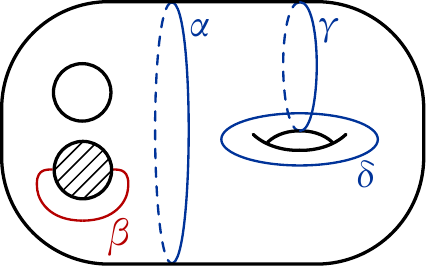} 
\caption{Curves in $N_{3,1}$ from case 5.}\label{F5} %
\end{center}
\end{figure}
Observe that
$\Lk(\beta)\cap\Lk(\gamma)\cap\Lk(\delta)=\{\alpha\}$. We have $f(\beta)=A(\beta)$ by case (4) and $f(\gamma)=A(\gamma)$, $f(\delta)=A(\delta)$  by case (1),  which implies
\begin{align*}
  \{f(\alpha)\}&=\Lk(f(\beta))\cap\Lk(f(\gamma))\cap\Lk(f(\delta))\\
  &=\Lk(A(\beta))\cap\Lk(A(\gamma))\cap\Lk(A(\delta))=\{A(\alpha)\}.  
\end{align*}
\end{proof}

\begin{rem}
Our proof of Theorem \ref{thm:aut_curv} is similar to the proof of the main result of \cite{AK}. In fact Step 1 of our proof is the same as in \cite[Theorem 7.7]{AK} (although more detailed in the case $n=1$). The main difference is in Step 2, where we use connectivity of the graph $\curv^1_e(N)$, which requires the assumption $g\ge 2$. We can assume $g\ge 2$ because the case of $N_{1,3}$ was settled in \cite{Szep} by a different argument. The proof in \cite{AK} is based on  connectivity of a different graph and it works also for $g=1$.
\end{rem}

%% file: N30.tex
\section{The graph $\pants(N_3)$.} \label{sec:N30}
In this section we focus on the case $N=N_3$. We describe the structure of the pants graph $\pants(N)$ and construct an abstract graph isomorphic to it. Additionally, we  define a simply-connected 2-dimensional complex $\overline{\pants(N)}$, whose
      1-skeleton is $\pants(N)$, and such that $\Aut(\overline{\pants(N)})=\Aut(\pants(N))$.

\subsection{Curves on $N$.}
There is a unique (up to isotopy) one-sided curve $\alpha_0$ in $N$ such that $N_{\alpha_0}\cong S_{1,1}$, see \cite{Sch}. 
The remaining curves in $N$ can be partitioned into two disjoint subsets:
\[\curv^0(N)\setminus\{\alpha_0\}=\curv^+\cup\curv^-,\]
where $\curv^+$ consists of  two-sided curves, and  $\curv^-$  consisting of  one-sided curves whose  complement is nonorientable.
    
For every $\alpha\in\curv^0(N)$, $\alpha\ne \alpha_0$ 
\[
N_{\alpha}\cong\begin{cases}
N_{1,2}&\alpha\in\curv^+\\
N_{2,1}&\alpha\in\curv^-
\end{cases}\qquad
i(\alpha,\alpha_0)=\begin{cases}
0&\alpha\in\curv^+\\
1&\alpha\in\curv^-.
\end{cases}
\]
\begin{lemma}\label{N30:bijection}
 There exists a bijection $\varphi\colon\curv^-\to\curv^+$ such that for $\alpha\in\curv^-$ and $\beta\in\curv^+$    
 \[i(\alpha,\beta)=0\iff \beta=\varphi(\alpha).\]
\end{lemma}
\begin{proof}
    For $\alpha\in\curv^-$, since $N_\alpha\cong N_{2,1}$, there exists a unique two-sided curve $\beta$ such that $i(\alpha,\beta)=0$. We define $\varphi(\alpha)=\beta$. To prove that $\varphi$ is a bijection, we construct its inverse $\varphi^{-1}\colon\curv^+\to\curv^-$. For $\beta\in\curv^+$, since $N_\alpha\cong N_{1,2}$, there exist two one-sided curves disjoint from $\beta$. One of them is $\alpha_0$, and  the other one is $\varphi^{-1}(\beta)$.
\end{proof}
\begin{lemma}\label{N30:index}
 For $\alpha,\alpha'\in\curv^-$     
 \[i(\alpha,\alpha')=0\iff i(\varphi(\alpha),\varphi(\alpha'))=1.\]
\end{lemma}
\begin{proof}
See the configuration of pairs $\{\alpha,\varphi(\alpha)\}$ and $\{\alpha',\varphi(\alpha')\}$ in Figure \ref{N30:dual}.
    \begin{figure}[h]
\begin{center}
\includegraphics[width=0.47\customwidth]{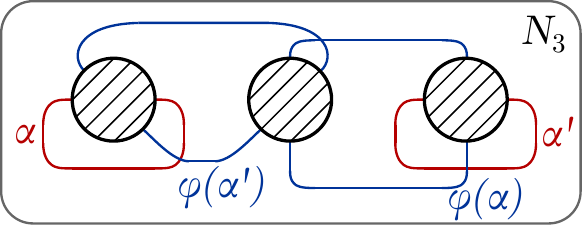} 
\caption{Curves $\alpha,\alpha',\varphi(\alpha)$ and $\varphi(\alpha)$ in $N_3$ -- Lemma \ref{N30:index}.}\label{N30:dual} %
\end{center}
\end{figure}
\end{proof}

\subsection{Vertices and edges of $\pants(N)$.} 
\begin{prop}
Every vertex of $\pants(N)$ belongs to one of the following sets.
\begin{align*}
    V_1&=\{\{\alpha_0,\beta\}\,\mid\,\beta\in\curv^+\},\\
    V_2&=\{\{\alpha,\varphi(\alpha)\}\,\mid\,\alpha\in\curv^-\}=\{\{\varphi^{-1}(\beta),\beta\}\,\mid\,\beta\in\curv^+\},\\
    V_3&=\{\{\alpha_1,\alpha_2,\alpha_3\}\,\mid\,\alpha_i\in\curv^-,\, i=1,2,3,\, i(\alpha_i,\alpha_j)=0\textrm{\ for\ }i\ne j\}.
\end{align*}  
\end{prop}
\begin{proof}
    Let $X\in\pants^0(N)$. If $X^+=\emptyset$ then $X\in V_3$. If $\beta\in X^+$ then since $N_\beta\cong N_{1,2}$, either $X=\{\beta,\alpha_0\}\in V_1$ or
    $X=\{\beta,\varphi^{-1}(\beta)\}\in V_2$.
\end{proof}
Let $\psi\colon V_1\to V_2$ be the bijection defined by
\[\psi(\{\alpha_0,\beta\})=\{\varphi^{-1}(\beta),\beta\}\]
for $\beta\in\curv^+$.  For $i=1,2$ we define maps $v_i\colon\curv^+\cup\curv^-\to V_i$ by the formulas
\[
v_1(\alpha)=\begin{cases}
\{\alpha_0,\alpha\}&\alpha\in\curv^+\\
\{\alpha_0,\varphi(\alpha)\}&\alpha\in\curv^-
\end{cases}\qquad
v_2(\alpha)=\begin{cases}
\{\varphi^{-1}(\alpha),\alpha\}&\alpha\in\curv^+\\
\{\alpha,\varphi(\alpha)\}&\alpha\in\curv^-.
\end{cases}
\]
Observe that $v_2=\psi\circ v_1$. For $i=1,2$ the restrictions  $v_i|_{\curv^+}$ and  $v_i|_{\curv^-}$ are bijections onto $V_i$, and for $\alpha\in\curv^-$, $\beta\in\curv^+$ we have
\[v_i(\alpha)=v_i(\beta)\iff \beta=\varphi(\alpha).\]

\begin{prop}
Every edge of $\pants(N)$ belongs to one of the following  sets.
\begin{align*}
    E_{11}&=\{v_1(\beta)\mvI v_1(\beta')\,\mid\,\beta,\beta'\in\curv^+,\ i(\beta,\beta')=1\},\\
    E_{12}&=\{X\mvIII\psi(X)\,\mid\,X\in V_1\},\\
     E_{23}&=\{X\mvIVL v_2(\alpha)\,\mid\, X\in V_3, \alpha\in X\},\\
    E_{33}&=\{\{\alpha_1,\alpha_2,\alpha_3\}\mvIII\{\alpha'_1,\alpha_2,\alpha_3\}\,\mid\,
    \{\alpha_1,\alpha_2,\alpha_3\},\{\alpha'_1,\alpha_2,\alpha_3\}\in V_3\}.
\end{align*}  
\end{prop}
\begin{proof}
Let $e=XY$ be an edge of $\pants(N)$. If $X,Y\in V_3$ then $e\in E_{33}$.
Suppose that $e$ is of the form \[\{\alpha,\beta\}\mv\{\alpha',\beta\},\]
where $\alpha,\beta$ are curves ($X\notin V_3$). If $\beta\in\curv^+$ then $\{\alpha,\alpha'\}=\{\alpha_0,\varphi^{-1}(\beta)\}$ and $e\in E_{12}$. If $\beta=\alpha_0$ then $e\in E_{11}$. Finally, suppose $\beta\in\curv^-$ and $\alpha=\varphi(\beta)$, so that $X=v_2(\beta)$. Then $\alpha'\notin\curv^+$, because $\alpha$ is the unique two-sided curve disjoint from $\beta$. Thus $\alpha'=\{\alpha_1,\alpha_2\}$ for $\alpha_1,\alpha_2\in\curv^-$, $Y\in V_3$ and $e\in E_{23}$.
\end{proof}
\begin{rem}
    Let $X=\{\alpha_1,\alpha_2,\alpha_3\}\in V_3$. Since $N_{\alpha_1}\cong N_{2,1}$, by the description of $\pants(N_{2,1})$ we have
    $i(\varphi(\alpha_1),\alpha_j)=1$ for $j=2,3$, which implies that $X$ is joined to $v_2(\alpha_1)=\{\alpha_1,\varphi(\alpha_1)\}\in V_2$ by a move of type 4. By symmetry, $X$ is also joined to $v_2(\alpha_2)$ and $v_2(\alpha_3)$. Since $X$ is also incident to three edges of type 3 (from $E_{33}$), $X$ is a vertex of degree $6$ -- see Figure \ref{fig:n3:star}. On the other hand, every vertex from $V_1$ is contained in a Farey subgraph, and every vertex from $V_2$ is the centre of a fan subgraph, so these vertices have infinite degree.
\end{rem}
\begin{rem}
Every edge from $E_{33}$ belongs to two different fan triangles -- see Figure \ref{PN30:full}. Namely, an edge of the form \[\{\alpha_1,\alpha_2,\alpha_3\}\mvIII\{\alpha'_1,\alpha_2,\alpha_3\}\]
belongs to triangles with the third vertex  $v_2(\alpha_j)=\{\alpha_j,\varphi(\alpha_j)\}$ for $j=2,3$. 
    \begin{figure}[h]
\begin{center}
\includegraphics[width=0.86\customwidth]{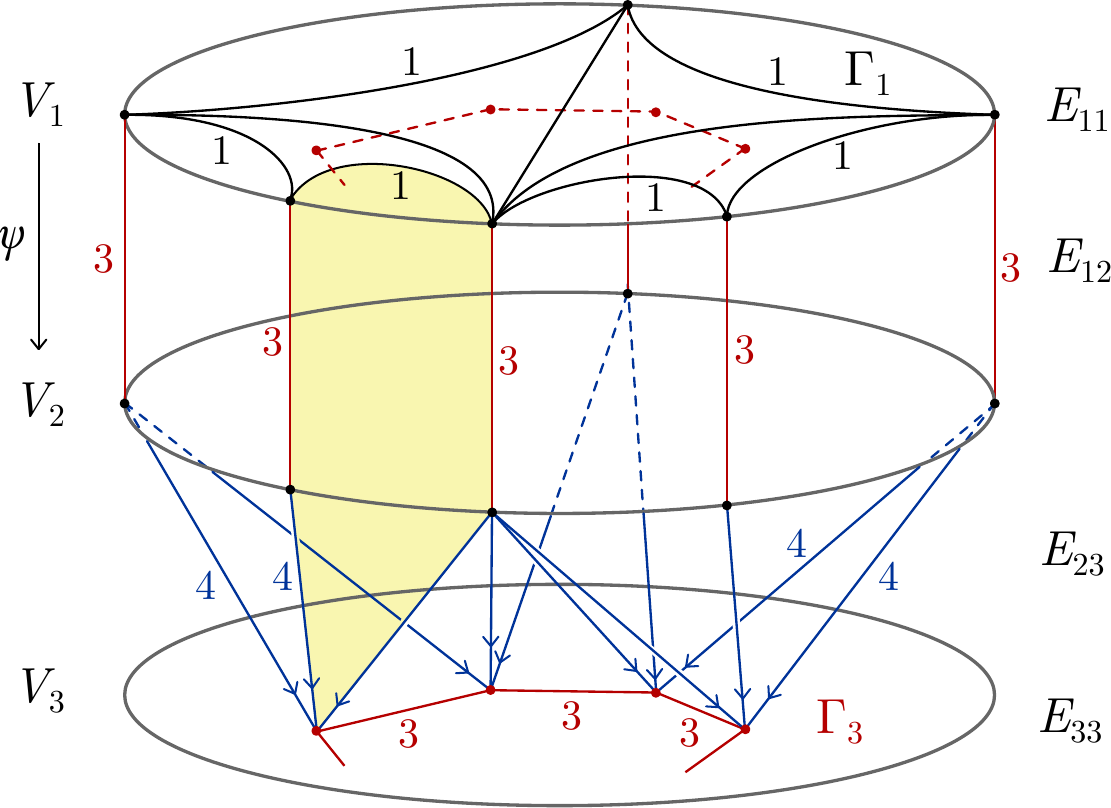} 
\caption{Vertices and edges of the pants graph $\pants(N_3)$. The shaded region is one of the standard pentagons -- see Definition \ref{def:std:pent}.}\label{PN30:full} %
\end{center}
\end{figure}
On the other hand, edges form $E_{12}$ are not contained in any triangles -- see Example \ref{ex:lonely:edge}.
\end{rem}
\begin{rem}
Using terminology from graph theory, the set $E_{12}$ is a \emph{matching}, that is to say no two edges in $E_{12}$ share a common vertex. It is a realisation of the bijection $\psi\colon V_1\to V_2$, as $X\in V_1$ is matched with $\psi(X)\in V_2$.
\end{rem}

\subsection{Subgraphs of $\pants(N)$.} Consider the following natural subgraphs of $\pants(N)$:
\[\Gamma_1=(V_1,E_{11}),\quad \Gamma_3=(V_3,E_{33}).\]
\begin{prop}
    $\Gamma_1$ is the unique Farey subgraph of $\pants(N)$.
\end{prop}
\begin{proof}
    We have \[\Gamma_1\cong\pants(N_{\alpha_0})\cong\pants(S_{1,1}),\]
    so $\Gamma_1$ is a Farey subgraph. Since $\Gamma_1$ contains all edges of type 1, and there are no edges of type 2 in $\pants(N)$, $\Gamma_1$ contains all Farey triangles of $\pants(N)$.
\end{proof}

\begin{defi}(Farey tree)
Farey tree $\Gamma_1^\ast$ is the graph which has a vertex for each triangle of $\Gamma_1$, and two vertices of $\Gamma_1^\ast$ are  joined by an edge if and only if the corresponding triangles share an edge.     
\end{defi}

As the name suggests, $\Gamma_1^\ast$ is a regular tree of degree 3. Let $\Tri(\Gamma_1)$ denote the set of triangles of $\Gamma_1$. In this section we identify a triangle with its set of vertices.  By Lemma \ref{N30:index}, for $\alpha_i\in\curv^-$, $i=1,2,3$
\[\{\alpha_1,\alpha_2,\alpha_2\}\in V_3\iff\{v_1(\alpha_1),v_1(\alpha_2),v_1(\alpha_3)\}\in\Tri(\Gamma_1).\]
Let us define $\theta\colon V_3\to\Tri(\Gamma_1)$ by 
 \[ \theta(\{\alpha_1,\alpha_2,\alpha_3\})=\{v_1(\alpha_1),v_1(\alpha_2),v_1(\alpha_3)\}.\]
 Since $v_1|_{\curv^-}$ is a bijection, $\theta$ is also a bijection. Furthermore, 
 for $X,Y\in V_3$, $XY\in E_{33}$ if and only if the triangles $\theta(X)$ and $\theta(Y)$ share an edge. Thus $\theta$ is  a graph isomorphism $\theta\colon\Gamma_3\to\Gamma_1^\ast$ and we have the following result.
\begin{prop} 
$\Gamma_3$ is isomorphic to the Farey tree $\Gamma_1^\ast$.
\end{prop}

\begin{lemma}
    For $X\in V_2$ and $Y\in V_3$,
    \[XY\in E_{23}\iff \psi^{-1}(X)\in\theta(Y).\]
\end{lemma}
\begin{proof}
 $XY\in E_{23}$ if and only if $X=v_2(\alpha)$ for some $\alpha\in Y$, in other words $X\in v_2(Y)$. We have
$v_2(Y)=\psi(v_1(Y))=\psi(\theta(Y))$,
so $X\in v_2(Y)\iff \psi^{-1}(X)\in\theta(Y).$
\end{proof}

\begin{prop}
    Every $X\in V_2$ is the centre of a unique fan subgraph of $\pants(N)$ with vertex set
    \[\{X\}\cup\{Y\in V_3\mid \psi^{-1}(X)\in\theta(Y)\}.\]
\end{prop}
\begin{proof}
    If $X=\{\varphi^{-1}(\beta),\beta\}\in V_2$, where $\beta\in\curv^+$, then $F=F(X,\beta)$ is the unique fan subgraph of $\pants(N)$ containing $X$. If $Y\ne X$, then $Y$ is a vertex of $F$ if and only if $XY\in E_4$, which is equivalent to 
     $\psi^{-1}(X)\in\theta(Y)$ by the previous lemma.
\end{proof}

\subsection{Abstract model of $\pants(N)$.} Using the results of the two previous subsections we can construct an abstract graph isomorphic to $\pants(N)$ starting from the Farey graph $\farey$ and the associated Farey tree $\farey^\ast$. Let $V(\farey)$ (resp. $V(\farey^\ast)$) denote the vertex set of $\farey$ (resp. $\farey^\ast$). Recall that $V(\farey^\ast)$ is the set of triangles of $\farey$.

We define $\mathcal{G}$ to be the graph obtained from the disjoint union $\farey\sqcup\farey^\ast$ by adding a new vertex $\widetilde{x}$ for every $x\in V(\farey)$ together with edges $\widetilde{x}x$ and $\widetilde{x}T$ for every $T\in V(\farey^\ast)$ such that $x\in T$. In follows from the results of the two previous subsections that $\mathcal{G}$ is isomorphic to $\pants(N)$, with $\farey$ and $\farey^\ast$ corresponding to $\Gamma_1$ and $\Gamma_3$ respectively, and for $x\in V(\farey)$ corresponding to $X\in V_1$, $\widetilde{x}$ corresponds to $\psi(X)\in V_2$.

\subsection{Attaching 2-cells to $\pants(N)$.} 
The analysis from the previous section indicates that there are natural cycles of length 5 (pentagons) in $\pants(N)$.
\begin{defi}\label{def:std:pent}
    For $X\in V_3$ and $\alpha_1,\alpha_2\in X$ we define a \emph{standard pentagon} to be the circuit
    \[X\mvIVL v_2(\alpha_1)\mvIII v_1(\alpha_1)\mvI v_1(\alpha_2)\mvIII v_2(\alpha_2)\mvIV X.\]
  \end{defi}
  \begin{example}
 See Figure \ref{fig:N30:pent} for the geometric realization of a standard pentagon in $N$.
    \begin{figure}[h]
\begin{center}
\includegraphics[width=0.62\customwidth]{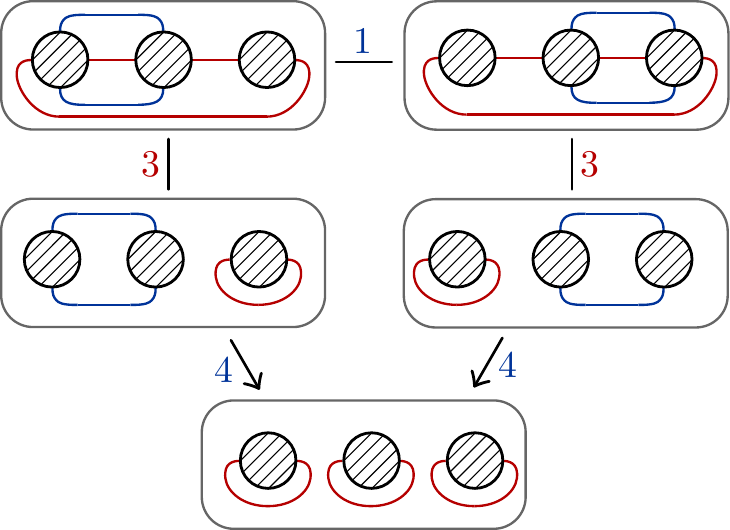} 
\caption{Standard pentagon in $N_3$.}\label{fig:N30:pent} %
\end{center}
\end{figure}
\end{example}
\begin{theorem}\label{prop:n30:2dim}
 The 2-dimensional complex  $\overline{\pants(N)}$ obtained  by attaching 2-cells to $\pants(N)$  along each triangle and each standard pentagon is simply-connected.
\end{theorem}
\begin{proof}
Let $L$ be a loop in $\overline{\pants(N)}$ based at $X_0\in V_1$. We can assume that $L$ is a simplicial loop in the graph $\pants(N)$. 
Since every edge from $E_{33}$ is contained in a fan triangle, every loop is homotopic in $\overline{\pants(N)}$ to a loop without edges from $E_{32}$, so we assume that $L$ does not contain edges from $E_{33}$.
Let $n$ be the length of $L$, i.e. the number of edges contained in $L$. We prove by induction on $n$ that $L$ is null-homotopic in $\overline{\pants(N)}$. If $n=0$ then $L$ is obviously null-homotopic, so assume $n\ge 1$ and that all loops of length less than $n$ are null-homotopic.

If $L$ is contained in the Farey subgraph $\Gamma_1$, then $L$ is null-homotopic, because the complex obtained by attaching 2-cells to $\Gamma_1$  along each Farey triangle  is simply-connected.

If $L$ is not contained in $\Gamma_1$, then it contains at least one vertex $X_1\in V_2$ and we can assume that $X_1$ is joined to $X_0$ by an edge from $E_{12}$ and this edge  is contained in $L$. If $n=2$ then $L$ travels from $X_0$ to $X_1$  and then back to $X_0$ along the same edge, so it is null-homotopic. If $n\ge 3$ then $L$ must contain $X_2\in V_3$, $X_3\in V_2$, and edges $X_1X_2\in E_{23}$, $X_2X_3\in E_{23}$ (because we assumed that $L$ contains no edges from $E_{33}$). Consider the following segment $L'$ of $L$:
\[X_0\mvIII X_1\mvIV X_2\mvIVL X_3.\]
If $X_3=X_1$ then by replacing $L'$ by the single edge $X_0X_1$ we obtain a loop of length $n-2$ homotopic to $L$, which is null-homotopic by induction hypothesis. If $X_3\ne X_1$ then $L'$ is contained in a standard pentagon
\[X_0\mvIII X_1\mvIV X_2\mvIVL X_3\mvIII Y\mvI X_0,\]
where $Y=\psi^{-1}(X_3)$. By replacing $L'$ by the segment
\[X_0\mvI Y\mvIII X_3,\]
we obtain a loop of length $n-1$ homotopic to $L$, which is null-homotopic by induction hypothesis.
\end{proof}